\documentclass[reqno,11pt]{amsart}

\usepackage{a4wide}
\usepackage{amsmath}
\usepackage{amssymb}
\usepackage{amsthm}
\usepackage{amsrefs}
\usepackage{graphicx}
\usepackage{epsfig,color}
\usepackage{epsfig}
\usepackage{tikz}

\def\Xint#1{\mathchoice
{\XXint\displaystyle\textstyle{#1}}%
{\XXint\textstyle\scriptstyle{#1}}%
{\XXint\scriptstyle\scriptscriptstyle{#1}}%
{\XXint\scriptscriptstyle\scriptscriptstyle{#1}}%
\!\int}
\def\XXint#1#2#3{{\setbox0=\hbox{$#1{#2#3}{\int}$ }
\vcenter{\hbox{$#2#3$ }}\kern-.6\wd0}}

\def\dashint{\Xint-}

\newcommand{\ind}{\protect\raisebox{2pt}{$\chi$}}

\newcommand{\di}{\text{d}}
\newcommand{\diam}{\text{diam}}
\newcommand{\dist}{\text{dist}}
\newcommand{\R}{\mathbb{R}}

\title{A DENSITY RESULT FOR HOMOGENEOUS SOBOLEV SPACES}

\begin{document}


\author[D. Nandi]{Debanjan Nandi}
\address[Debanjan Nandi]{Department of Mathematics and Statistics, University of Jyv\"askyl\"a, P.O. Box 35, FI-40014 Jyv\"askyl\"a, Finland}
\email{debanjan.s.nandi@jyu.fi}

\subjclass[2000]{}
\thanks{This research has been supported by the Academy of Finland via the Centre of Excellence in Analysis and Dynamics Research (project No. 307333).}

\theoremstyle{plain}
\newtheorem{thm}{Theorem}[section]
\newtheorem{lem}[thm]{Lemma}
\newtheorem{prop}[thm]{Proposition}
\newtheorem{cor}[thm]{Corollary}

\theoremstyle{definition}
\newtheorem{defn}[thm]{Definition}
\newtheorem{exm}[thm]{Example}
\newtheorem{prob}[thm]{Problem}

\theoremstyle{remark}
\newtheorem{rem}[thm]{Remark}

\begin{abstract}
We show that in a bounded Gromov hyperbolic domain $\Omega$ smooth functions with bounded derivatives $C^\infty(\Omega)\cap W^{k,\infty}(\Omega)$ are dense in the homogeneous Sobolev spaces $L^{k,p}(\Omega)$.
\end{abstract}

\maketitle

\section{Introduction}
We continue the study of density of functions with bounded derivatives in the space of Sobolev functions in a domain in $\R^n$. 
It was shown by Koskela-Zhang \cite{KZ} that for a  simply connected planar domain $\Omega\subset\mathbb{R}^2$, $W^{1,\infty}$ is dense in $W^{1,p}$ and in the special case of Jordan domains also $C^{\infty}(\mathbb{R}^2)\cap W^{1,\infty}(\Omega)$ is dense. The above result of Koskela-Zhang has been generalized to have the density of $W^{k,\infty}$ in the homogeneous Sobolev space $L^{k,p}$ for $k\in\mathbb{N}$ in planar simply connected domains by Nandi-Rajala-Schultz \cite{NRS}. In dimensions higher than two however simply connectedness is not sufficient (see for example \cite{KRZ}). Recall that simply connected planar domains are negatively curved in the (quasi) hyperbolic metric. A useful metric generalization of negatively curved spaces was introduced by Gromov \cite{Gr}, in the context of group theory. Following Bonk-Heinonen-Koskela \cite{BHK}, we call a domain Gromov hyperbolic if, when equipped with the quasihyperbolic metric, it is $\delta$-hyperbolic in the sense of Gromov, for some $\delta\geq 0$ (see Section \ref{prelims} for definitions). Gromov hyperbolicity has turned out to be a sufficient condition for the density of $W^{1,\infty}$ in $W^{1,p}$, as shown by Koskela-Rajala-Zhang \cite{KRZ}, and these are primarily the domains we consider in this paper.

Let us mention here a few similarities between our setting and the planar simply connected case. We recall that simply connected domains with the quasihyperbolic metric (equivalent to the hyperbolic metric by the Koebe distortion theorem) in the plane are Gromov hyperbolic and that the quasihyperbolic geodesics are unique (see Luiro \cite{Luiro}); conversely, a Gromov hyperbolic domain with uniqueness of quasihyperbolic geodesics is simply connected. The latter is of course true in higher dimensions as well (indeed, in this case, for any pair of points, any curve $\gamma$ joining the points is homotopic to the unique quasihyperbolic geodesic $\Gamma$ joining the given points, with the homotopy  given by quasihyperbolic geodesics joining the points $\gamma(t)$ and $\Gamma(t)$ once $\Gamma$ is parametrized suitably). Gromov hyperbolic domains are often seen as a topological generalization of planar simply connected domains. It was shown by Bonk-Heinonen-Koskela \cite{BHK} that Gromov hyperbolic domains are conformally equivalent to suitable uniform metric spaces equipped with the quasihyperbolic metric and corresponding suitable measures.  Finally, we note that in higher dimensions, simply connectedness alone does not imply Gromov hyperbolicity; consider for example the unit ball in $\mathbb{R}^3$ deformed to have a cuspidal-wedge along an equator.



We denote by $L^{k,p}$ the space of functions with finite homogeneous Sobolev norm (see Section \ref{prelims} for definition). We obtain the following extension of the result of Koskela-Rajala-Zhang \cite{KRZ} to Sobolev spaces of higher order.
 \begin{thm}\label{main_cor}
  Let $0\leq \delta<\infty$, $k\geq 1$ and $1\leq p <\infty$. If $\Omega\subset \R^n$ is a bounded $\delta$-Gromov hyperbolic domain, then $ C^{\infty}(\Omega)\cap W^{k,\infty}(\Omega)$ is dense in $L^{k,p}(\Omega)$.
 \end{thm}
We have the following corollary to Theorem \ref{main_cor}.
\begin{cor}\label{main_cor_2} 
  Let $0\leq \delta<\infty$, $k\geq 1$ and $1\leq p <\infty$. If $\Omega$ is a bounded $\delta$-Gromov hyperbolic domain with $C^0$-boundary, then $ C^{\infty}(\R^n)\cap W^{k,\infty}(\Omega)$ is dense in $L^{k,p}(\Omega)$.
\end{cor} In the statement of Corollary \ref{main_cor_2}, for a domain with $C^0$-boundary, we require that given $x\in\partial \Omega$, there is a neighbourhood $U_x\subset\R^n$ of $x$ such that $U_x\cap \Omega$ has the representation $y_n<f(y_1,\ldots,y_{n-1})$ in suitable coordinates, with a continuous function $f$ (see Maz'ya \cite{Mazya}).

 The approximating functions are obtained by an inner extension of a smooth approximation to the function to be approximated, from an increasing sequence of suitable compact subsets, to the domain in question. This method of inner extension for this purpose has been used in \cite{KZ}, \cite{KRZ} and \cite{NRS}. In \cite{KZ} this extension is done using the conformal parametrization of simply connected domains. In \cite{NRS}, the topology of the plane is utilized directly to construct an approximating sequence although the hyperbolic structure of planar simply connected domains plays an implicit role in the construction. Our approach in this paper is along the lines of \cite{KRZ}. 
 
 The information coming from the conformal invariance may be read off in terms of any suitable Whitney decomposition of our domain and towards that end we use as tools the uniformization of Gromov hyperbolic domains developed in \cite{BHK} along with two of its geometric implications also from \cite{BHK}, employed already in \cite{KRZ}, namely the ball separation condition and the Gehring-Hayman condition (see Section \ref{prelims} below for defintions). Using uniformization we prove a diameter counterpart of the usual Gehring-Hayman condition (see Theorem \ref{diameter_GeHa}) which in its original form relates the lengths of quasihyperbolic geodesics with the intrinsic distance between the points they join (see Pommerenke \cite{Po} for the property in simply connected domains).  We use also the polynomial approximations of smooth functions in precompact subsets of the domain (see Section \ref{approx_poly}) which were utilized by Jones \cite{J} for showing the extendibility of Sobolev functions defined in uniform domains and also in \cite{NRS} for the inner extension. We actually prove the density under weaker assumptions than Gromov hyperbolicity (see Theorem \ref{main_1}).

One may want to know under what other conditions such density results could hold. Recall that a domain $\Omega$ is called $c$-\textit{John with center} $x_0$ if for each $x\in\Omega$ there exists a curve $\gamma_x$ joining $x$ to $x_0$ such that at each point $z\in\gamma_x$ of the curve, $c$ times the distance to the boundary is larger than the length of the subcurve joining $x$ to $z$. The curve $\gamma_x$ is called a $c$-\textit{John curve}. The domain $\Omega$ is $c$-quasiconvex if for each pair of points $x,y\in\Omega$, $c\,\text{d}(x,y)\geq \lambda_{\Omega}(x,y)$. Here $\lambda_{\Omega}$ is the intrinsic distance; see Section 2.
There are domains that are simultaneously John and quasiconvex for which the density of $C^{\infty}\cap W^{k,\infty}$-functions in the Sobolev classes $W^{k,p}$ in the homogeneous  $L^{k,p}$-norm fails to hold. As an example, we recall the following class of domains which appear in a paper of Koskela \cite{Koskela}.
\begin{exm}\label{example}
 Let $E\subset\mathbb{R}^{n-1}$ be a $p$-porous set for $1<p\leq n$. Then $E$ is removable for $W^{1,p}$ in $\mathbb{R}^n$. Moreover, for each $1<p\leq n$, there exists a $p$-porous set $E\subset\mathbb{R}^{n-1}$ that is non-removable for $W^{1,q}$ for any $q<p$.
\end{exm}
Here a set $E$ is called removable for the class $W^{1,p}(\mathbb{R}^n)$ if we have that the norm preserving restriction operator $W^{1,p}(\mathbb{R}^n)\hookrightarrow W^{1,p}(\mathbb{R}^n\backslash E)$ is an isomorphism. If $E$ is removable for $W^{1,p}$ then it is removable for $W^{1,q}$ for $q>p$ by the H\"older inequality and since a set being removable is a local condition. The complements of sets removable for $W^{1,q}$ for $q>n$ are quasiconvex; see for example Koskela-Reitich \cite{KR}. Therefore for $p\leq n-1$, the unit ball with a suitable $(p+1)$-porous set as in Example \ref{example} removed from the intersection of a $(n-1)$-hyperplane with the concentric ball of radius half, is a John and quasiconvex domain, for which the $C^{\infty}\cap W^{1,q}$ functions, for $q>p+1$, can not be dense in the class $W^{1,p}$. The Sobolev-Poincare inequality holds in John domains (see Bojarski \cite{Bo}) from which it follows that $L^{1,p}=W^{1,p}$ in John domains and in particular that $C^{\infty}\cap W^{1,q}$ is not dense in $L^{1,p}$ for $q>p+1$.

We note however that domains where the John curves may be chosen to be quasihyperbolic geodesics and a suitable quasiconvexity condition involving the lengths of the quasihyperbolic geodesics holds, are Gromov hyperbolic (see Section \ref{geometric_definitions}) and the required density therefore follows. 

We say that a domain $\Omega$ admits a global $p$-Poincar\'e inequality if for locally integrable functions $u\in L^1_{loc}(\Omega)$ with (weak) first order $p$-integrable derivatives it holds$$\int_{\Omega} |u-u_{\Omega}|^p\leq C\int_{\Omega}|\nabla u|^p,$$ for $C=C(\Omega)$, where $u_{\Omega}=\dashint_{\Omega} u$. We have the following corollary to Theorem \ref{main_cor}.
\begin{cor}\label{main_cor3}
 Let $0\leq \delta<\infty$, $k\geq 1$ and $1\leq p <\infty$. If $\Omega\subsetneq\R^n$ is a $\delta$-hyperbolic domain which admits a global $p$-Poincar\'e inequality, then $W^{k,\infty}(\Omega)$ is dense in $W^{k,p}(\Omega)$. 
\end{cor}

The paper is arranged as follows. In Section \ref{notation}, we set some of the notation to be used below. In Section \ref{prelims}, we introduce definitions of relevant function spaces and geometric conditions on domains being used in this paper. In Section \ref{homogeneous_density}, we prove the density result under the (weaker) geometric conditions from Section \ref{prelims}. In Section \ref{Hyperbolic}, we verify that Gromov hyperbolic domains satisfy the geometric assumptions sufficient for the density result which proves Theorem \ref{main_cor} and then prove some technical facts required for the proof of Theorem \ref{main_1} in Section \ref{homogeneous_density}.

\section{Notation}\label{notation}
We write $\#A$ for the cardinality of a set $A$ and the Lebesgue $\mathcal{L}^n$-measure of sets $A\subset\mathbb{R}^n$ is denoted $|A|$. For an open cube $Q$ with edges parallel to the coordinate axes, $l(Q)$ denotes its edge length. For a point $x\in\Omega$ we write $\text{d}_{\Omega}(x)$ for $\text{d}(x,\partial\Omega)$.
For a set $A\subset \R^n$ and $\epsilon>0$, $B(A,\epsilon)$ is the $\epsilon$-neighbourhood of $A$ in $\R^n$. When $A=\{x\}$, we just write $B(x,r)$.

The intrinsic length metric of a domain $\Omega$, denoted $\lambda_{\Omega}(x,y)$, is the infimum of the euclidean lengths of paths in $\Omega$ joining a given pair of points $x,y\in\Omega$. The intrinsic diameter metric is the infimum of euclidean diameters of paths in $\Omega$ joining a pair of points and is denoted $\delta_{\Omega}(x,y)$. For a curve $\gamma\subset\Omega$ (with injective parametrization) and $x,y\in\gamma$ we write $\gamma(x,y)$ for the subcurve of $\gamma$ between the points $x$ and $y$. We write $l(\gamma)$ for the length of a rectifiable curve $\gamma$.

For $x\in\Omega$, we write $B_{\Omega}(x,\epsilon)$, for the set of points $y$ in $\Omega$ such that $\lambda_{\Omega}(x,y)< \epsilon$, that is the $\epsilon$ ball in the intrinsic metric of $\Omega$. Similarly for sets $A\subset\Omega$ we write $B_{\Omega}(A,\epsilon)$ for the intrinsic $\epsilon$-neighbourhoods. We use $\dist(\cdot,\cdot)$ to denote the distance between sets obtained by taking infimum of pairwise distances of points lying in the respective sets. 

For a set $A\subset\Omega$, the domain $\Omega$ in question being fixed, we will write $\bar{A}$ for the closure of $A$ in the relative topology of $\Omega$.

\section{Preliminaries}\label{prelims}

Let $\Omega\subset\R^n$. Let $p\in[1,\infty]$ and $k\in\mathbb{N}$. We write $L^{k,p}(\Omega)$ for the space of Sobolev functions with $p$-integrable distributional derivatives of order $k$;
$$L^{k,p}(\Omega)=\{u\in L^1_{loc}(\Omega):D^{\alpha}u\in L^{p}(\Omega),\;\text{if}\;|\alpha|=k\}.$$ We equip it with the homogeneous Sobolev seminorm $\sum_{|\alpha|=k}\|\nabla^{\alpha} u\|_{L^p(\Omega)}$.
The (non-homogeneous) Sobolev space $W^{k,p}(\Omega)$ is defined as 
$$W^{k,p}(\Omega)=\{u\in L^1_{loc}(\Omega):D^{\alpha}u\in L^{p}(\Omega),\;\text{if}\;|\alpha|\leq k\}$$ and is equipped with the norm $\sum_{|\alpha|\leq k}\|\nabla^{\alpha} u\|_{L^p(\Omega)}$. Here and below an $n$-multi-index $\alpha$ is an $n$-vector of non-negative integers and $|\alpha|$ is its $\ell_1$-norm.

\subsection{Whitney decomposition}
We use the standard Whitney decomposition of a domain $\Omega\subset\R^n$ (see for instance Whitney \cite{whitney} or the book of Stein \cite[Chapter 6]{stein}).

\subsection{Some geometric conditions}\label{geometric_definitions}

\begin{defn}[Uniform domain]\label{uniformdomains}
 A domain $\Omega\subsetneq\R^n$ is $A$-uniform if there exists a constant $A\geq 1$ such that for each pair of points $x,y\in\Omega$, there exists a curve $\gamma_{xy}\subset\Omega$ joining $x$ and $y$ such that 
 \begin{enumerate}
  \item [(\textit{i})]$l(\gamma_{xy})\leq A|x-y|,$
  \item [(\textit{ii})]For any $z\in\gamma_{xy}$, $l(\gamma_{xy}(x,z)\wedge l(\gamma_{xy}(z,y)))\leq A\di_{\Omega}(z).$ 
 \end{enumerate}The curves $\gamma_{xy}$ are called uniform curves. Curves that satisfy only the second requirement are called doubly-John. 
\end{defn} 
Uniform domains, introduced by Martio-Sarvas \cite{MarSar} are more general than Lipschitz domains but nice enough for being Sobolev extension domains (see Jones \cite{J}). Therefore Lipschitz functions are dense both in the homogeneous and non-homogeneous Sobolev spaces defined in these domains. They come up naturally in the theory of quasiconformal mappings, for example as a characterization for the quasisymmetric images of disks. Note that the definition requires only a non-complete metric space. We may similarly define uniform spaces as metric spaces with non empty topological boundary that satisfy the conditions above. We will require the notion of uniform spaces in Section \ref{Hyperbolic}.

\begin{defn}(Quasihyperbolic metric)\label{defn_qhmetric}
Let $\Omega\subsetneq \mathbb{R}^n$ be a domain. Given $x,y\in\Omega$ the quasihyperbolic distance is defined as
\begin{equation*}k_{\Omega}(x,y):=\underset{\gamma_{xy}}{\inf} \int \frac{|dz|}{d_{\Omega}(z)} \end{equation*} where the infimum is taken over all rectifiable curves joining $x$ and $y$ in $\Omega$. 
\end{defn}
It is a consequence of the Arzela-Ascoli theorem that a quasihyperbolic geodesic (for which the infimum is achieved) exists for each pair of points; see for example the paper of Gehring-Osgood \cite{GeOs} where several facts about the quashyperbolic metric and its relation with quasiconformal mappings had been proved. It was also shown in \cite{GeOs} that in uniform domains 
\begin{equation}\label{quasihyperbolic_growth}
  k_{\Omega}(x,y)\simeq \log\left(1+\frac{|x-y|}{\di_{\Omega}(x)\wedge \di_{\Omega}(y)}\right)
\end{equation}for $x,y\in\Omega$.
It may be noted that quasihyperbolic distance dominates the logarithmic term in (\ref{quasihyperbolic_growth}) in general, that is 
 \begin{equation}\label{quasihyperbolic_growth_2}
  k_{\Omega}(x,y)\geq \log\left(1+\frac{\lambda_{\Omega}(x,y)}{\di_{\Omega}(x)\wedge \di_{\Omega}(y)}\right)
\end{equation}for $x,y\in\Omega$, where $\Omega$ need not be uniform. We use this and its consequence 
\begin{equation}
 k_{\Omega}(x,y) \geq \left| \log{\left(\frac{d_{\Omega}(x)}{d_{\Omega}(y)}\right)}\right|,
\end{equation}
later in Section \ref{Hyperbolic}.
 We note that the equivalence in \eqref{quasihyperbolic_growth} characterises uniformity (see \cite{GeOs}). We define here two more quantities (which are equivalent to the logarithmic term in \eqref{quasihyperbolic_growth} in uniform domains but not in general), which we will require later. Set $$\Delta_{\Omega}(x,y)=\log\left(1+\frac{\delta_{\Omega}(x,y)}{\di_{\Omega}(x)\wedge \di_{\Omega}(y)}\right)$$ and $$\Lambda_{\Omega}(x,y)=\log\left(1+\frac{\lambda_{\Omega}(x,y)}{\di_{\Omega}(x)\wedge \di_{\Omega}(y)}\right),$$ for $x,y\in\Omega\subsetneq\R^n$. These quantities are quasi-metrics in uniform domains.

\begin{defn}[Ball separation property]
 Let $\Omega\subsetneq\R^n$ be a domain. A curve $\Gamma\subset\Omega$ satisfies the $c$-ball separation property, for some $c\geq 1$, if for every point $z\in\Gamma$, the intrinsic ball $B=B_{\Omega}(z,c\di_{\Omega}(z))$ separates any  $x,y\in\Gamma\setminus B$ in $\Omega$ such that $z\in\Gamma(x,y)$. The domain $\Omega$ has the $c$-ball separation property if every quasihyperbolic geodesic in $\Omega$ satisfies the $c$-ball separation property.
\end{defn}
The above separation property was perhaps first introduced by Buckley-Koskela \cite{KosBuck} who showed that in the class of domains satisfying the above property, if the domain admits a global $(np/(n-p),p)$-Sobolev-Poincar\'e inequality for $1\leq p<n$, then the separation condition of geodesics improves to geodesics being John.
If a domain $\Omega$ is the quasiconformal image of a uniform domain, then the quasihyperbolic geodesics of $\Omega$ are the images of diameter-uniform curves (that is properties (\textit{i}) and (\textit{ii}) in Definition \ref{uniformdomains} hold with length replaced by diameter; existence of such curves for each pair of points is quantitatively equivalent to the definition of uniformity given here, see Martio-Sarvas \cite{MarSar} for this) under the quasiconformal mapping, see Heinonen-N\"akki \cite{HeiNak}. An argument using the conformal modulus (see \cite{KosBuck}) then shows that quasiconformal images of diameter uniform curves (and thus the quasihyperbolic geodesics in quasiconformal images of uniform domains in particular) have the ball separation property.

\begin{defn}[Gehring-Hayman property]
Let $\Omega\subsetneq \R^n$ be a domain. A curve $\Gamma\subset\Omega$ satisfies the $c$-Gehring-Hayman property, for $c\geq 1$, if $$l(\Gamma(x,y))\leq c\lambda_{\Omega}(x,y),$$ for all $x,y\in\Gamma$. The domain $\Omega$ has the $c$-Gehring-Hayman property if every quasihyperbolic geodesic $\Gamma\subset\Omega$ satisfies the $c$-Gehring-Hayman property. 
\end{defn}
The Gehring-Hayman property appears first perhaps in the paper \cite{GH} where it was shown to hold for hyperbolic geodesics in simply connected domains in the plane. Subsequently, it was shown to hold for quasiconformal images of uniform domains by Heinonen-Rohde \cite{HeiRoh} and for conformal metric deformations of the euclidean unit ball by Bonk-Koskela-Rohde \cite{BonkKosRoh}). We will define below a local version of this property which will suffice for our purpose.

By a Gromov hyperbolic domain we mean a domain $\Omega$ such that the space $(\Omega,k_{\Omega})$ is Gromov hyperbolic.
\begin{defn}[Gromov hyperbolicity]
 Let $0\leq \delta<\infty$. A domain $\Omega\subsetneq\R^n$ is $\delta$-Gromov hyperbolic if all quasihyperbolic geodesic triangles are $\delta$-thin in the quasihyperbolic metric. That is, given any three points $x,y,z\in\Omega$ and quasihyperbolic geodesics $\Gamma_{xy},\Gamma_{yz}$ and $\Gamma_{zx}$ joining them pairwise, it holds  for any $w$ lying in any of the three geodesics that the ball of radius $\delta$, in the quasihyperbolic metric, centered at $w$ intersects the union of the remaining two geodesics. 
\end{defn}
Quasiconformal images of uniform domains are Gromov hyperbolic (for example by the results in \cite{HeiRoh} and \cite{KosBuck} combined with Theorem \ref{Gromovhyperbolic_implies} below). A strictly weaker version of uniformity obtained by using the intrinsic metric of the domain instead of the euclidean metric, for the quasiconvexity of the uniform curve, is often useful. Domains satisfying this latter property are called inner uniform. 
Inner uniform domains are Gromov hyperbolic (see \cite{BHK}), and thus functions with bounded derivatives are dense in the homogeneous Sobolev spaces in these domains by Theorem \ref{main_1}. In fact, the density in this case holds also in the non-homogeneous Sobolev norm as the hypotheses of Corollary \ref{main_cor3} are satisfied (see for example \cite{Bo}).
The next theorem is a known characterization of Gromov hyperbolic domains. The geometric implications were proved in \cite{BHK} using uniformization. It was later shown by Balogh-Buckley \cite{BB} that they characterize Gromov hyperbolicity. 
\begin{thm}[\cite{BHK},\cite{BB}]\label{Gromovhyperbolic_implies}
A $\delta$-Gromov hyperbolic domain has both the $c_1$-ball separation property and the $c_2$-Gehring-Hayman property, for $c_1=c_1(\delta,n)$ and $c_2=c_2(\delta,n)$. Conversely, a domain which has both the $c_1$-ball separation property and the $c_2$-Gehring-Hayman property is also $\delta$-Gromov hyperbolic, for $\delta=\delta(c_1,c_2,n)$. 
\end{thm}

We define the following local versions of the Gehring-Hayman condition defined above which are suitable to our purpose. Recall the definitions of $\Lambda_{\Omega}(\cdot,\cdot)$ and  $\Delta_{\Omega}(\cdot,\cdot)$ from the discussion following Definition \ref{defn_qhmetric}.


 \begin{defn}[Local length/diameter Gehring-Hayman]\label{local_GeHa}
 Let $\Omega\subsetneq\R^n$ be a domain and $c\geq 1, R\geq 0$ be fixed. Let $E\subset\Omega$ be given. We say that $E$ has the $(c,R)$-length Gehring-Hayman property if for any pair of points $x,y\in E$  for which  
 $$\Lambda_{\Omega}(x,y)\leq R,$$ each quasihyperbolic geodesic $\Gamma_{xy}$ joining $x$ and $y$ in $\Omega$ satisfies
 \begin{equation*}
 l(\Gamma_{xy})\leq c\,\lambda_{\Omega}(x,y).
\end{equation*}
 We say that $E$ has the $(c,R)$-diameter Gehring-Hayman property if for any pair of points $x,y\in E$ for which
 $$\Delta_{\Omega}(x,y)\leq R,$$  
 each quasihyperbolic geodesic $\Gamma_{xy}$ joining $x$ and $y$ in $\Omega$ satisfies
\begin{equation*}
 \diam(\Gamma_{xy})\leq c\,\delta_{\Omega}(x,y).
\end{equation*}
\end{defn}

\begin{rem}\label{alternate_defn}
 The above definition may be equivalently reformulated by saying $E\subset\Omega$ has the above $(c,R)$-length Gehring-Hayman property if for all $x,y\in E$, for which \begin{equation}\label{hyperbolicballs}\frac{1}{R}\,\di_{\Omega}(x)\leq \di_{\Omega}(y)\leq R\,\di_{\Omega}(x)\end{equation} and  \begin{equation}\lambda_{\Omega}(x,y)\leq R\,(\di_{\Omega}(x)\wedge \di_{\Omega}(y)),\end{equation} each quasihyperbolic geodesic $\Gamma_{xy}$ joining $x$ and $y$ in $\Omega$ satisfies
 \begin{equation*}
 l(\Gamma_{xy})\leq c\,\lambda_{\Omega}(x,y).
\end{equation*}
 Similarly, we may say that $E$ has the $(c,R)$-diameter Gehring-Hayman property if 
\begin{equation*}
 \diam(\Gamma_{xy})\leq c\,\delta_{\Omega}(x,y)
\end{equation*}holds for any pair of points $x,y\in E$ and quasihyperbolic geodesics $\Gamma_{xy}$ joining them, where $x$ and $y$ satisfy the condition (\ref{hyperbolicballs}) of being comparably distant from the boundary as above and $$\delta_{\Omega}(x,y)\leq R\,(\di_{\Omega}(x)\wedge \di_{\Omega}(y)).$$ Note that the value of $R$ in this reformulation is related to but different from the $R$ in the previous defintion. 
\end{rem}

The first condition (\ref{hyperbolicballs}) says for example that quasihyperbolic unit balls centred at $x$ and $y$ have comparable sizes in the euclidean geometry.
It is not known if the $(c,R)$-length and diameter Gehring-Hayman conditions defined above are quantitatively equivalent for domains which have the ball separation property. 
However, in the case of domains where quasihyperbolic geodesics are doubly-John curves (see Definition \ref{uniformdomains}) the equivalence holds in the global sense. A domain is $(c,\infty)$-diameter Gehring-Hayman, if it is a $(c,R)$-diameter Gehring-Hayman domain for all $R>0$.

\begin{lem}\label{johndialen}
 Let $\Omega\subset\R^n$ be a domain such that the quasihyperbolic geodesics are $A$-doubly-John curves. Suppose $\Omega$ has the $(c_0,\infty)$-diameter Gehring-Hayman property. Then $\Omega$ also has the $c$-Gehring-Hayman property, where $c=c(A,c_0,n)$, that is, $\Omega$ is inner uniform. The converse is also true quantitatively. 
\end{lem}
\begin{proof}
 Let $\Gamma$ be a quasihyperbolic geodesic joining $x,y\in\Omega$. 
Let $z\in\Gamma$ be the midpoint, that is $l(\Gamma(x,z))=l(\Gamma(z,y))=l(\Gamma)/2$. We may assume that $\delta_{\Omega}(x,y)\geq \di_{\Omega}(z)/2$. Then by the John property of $\Gamma$, we have
 $$3A\delta_{\Omega}(x,y)\geq A\di_{\Omega}(z)\geq l(\Gamma(x,z)).$$ 
 The converse follows from Lemma \ref{diameter_GeHa}.
\end{proof}

The converse utilizes conformal uniformization (see Section \ref{Hyperbolic}) and the proof can not be applied if the global inequalities are replaced by the local ones. Lemma \ref{local_dia_implies_len_2} below says that if there exists any curve that joins $x$ and $y$ and lies uniformly away from the boundary then the local diameter Gehring-Hayman property implies the local length Gehring-Hayman property for the pair $\{x,y\}$. 
We need to state the following simple lemma before Lemma \ref{local_dia_implies_len_2}.

\begin{lem}\label{local_dia_implies_len}
 Let $M>0$ be a given number. Suppose $\Omega$ has the $(c,R)$-diameter Gehring-Hayman property and the $c_0$-ball separation property. Let $x,y\in\Omega$ be such that 
  $$\frac{1}{R}\di_{\Omega}(x)\leq \di_{\Omega}(y)\leq R\di_{\Omega}(x)$$ and  $$\delta_{\Omega}(x,y)\leq R(\di_{\Omega}(x)\wedge \di_{\Omega}(y)).$$
 Suppose there exists a curve $\gamma\subset\Omega$ joining $x$ and $y$ such that for each $z\in\gamma$, $\di_{\Omega}(z)\geq (\di_{\Omega}(x)\wedge \di_{\Omega}(y))/M$. Then $$k_{\Omega}(x,y)\leq c'(c,R,M,n).$$ 
\end{lem}
\begin{proof}
 Assume that $\delta_{\Omega}(x,y)\leq  R(\di_{\Omega}(x)\wedge \di_{\Omega}(y))$. Then we have by the $(c,R)$-diameter Gehring-Hayman property that $\diam(\Gamma)\leq cR \di_{\Omega}(x)$ for any quasihyperbolic geodesic $\Gamma$ joining $x$ and $y$. Fix such a geodesic $\Gamma$. The ball separation condition implies that, for any $z\in\Gamma$, there exists a $z'\in\gamma$ such that $\lambda_{\Omega}(z,z')\leq c_0\di_{\Omega}(z)$. Along with our assumption for the curve $\gamma$, it follows that $$\di_{\Omega}(z)\geq \frac{1}{M(c_0+1)}\di_{\Omega}(x),$$ and a similar estimate then follows for the edge-lengths of the Whitney cubes intersecting $\Gamma$. Since all these cubes lie in a ball of radius at most $2R \di_{\Omega}(x)$, by the upper bound on the diameter of $\Gamma$, there are at most $c'(c,R,M,n)$ of these. This gives an upper bound for $k_{\Omega}(x,y)$. 
\end{proof}

\begin{lem}\label{local_dia_implies_len_2}
 Let $M>0$ be a given number. Suppose $\Omega$ has the $(c,R)$-diameter Gehring-Hayman property and the $c_0$-ball separation property. Let $x,y\in\Omega$ be such that 
  $$\frac{1}{R}\di_{\Omega}(x)\leq \di_{\Omega}(y)\leq R\di_{\Omega}(x)$$ and  $$\lambda_{\Omega}(x,y)\leq R(\di_{\Omega}(x)\wedge \di_{\Omega}(y)).$$
 Suppose there exists a curve $\gamma\subset\Omega$ joining $x$ and $y$ such that for each $z\in\gamma$, $\di_{\Omega}(z)\geq (\di_{\Omega}(x)\wedge \di_{\Omega}(y))/M$. Then, $$l(\Gamma_{xy})\leq c(c,R,M,n)\lambda_{\Omega}(x,y)$$ where $\Gamma_{xy}$ is any quasihyperbolic geodesic joining $x$ and $y$ in $\Omega$. 
\end{lem}
\begin{proof}
 Since $\delta_{\Omega}(x,y)\leq\lambda_{\Omega}(x,y)$, we have by the previous lemma that $k_{\Omega}(x,y)\leq c'(c,R,M,n)$. 
 
 We may assume that $\lambda_{\Omega}(x,y)\geq \frac{1}{2}\di_{\Omega}(x)$. Fix a quasihyperbolic geodesic $\Gamma$ joining $x$ and $y$. Note that $l(\Gamma\cap Q)\leq 5l(Q)$, for each Whitney cube $Q$; see \cite{HeiRoh} for example. Thus, we get that $$l(\Gamma)\leq c'\cdot5\cdot 2R\di_{\Omega}(x)\leq 20c'R\lambda_{\Omega}(x,y)$$since the number of Whitney cubes intersecting $\Gamma$ is comparable with $k_{\Omega}(x,y)$ and their sizes are comparably smaller that $\di_{\Omega}(x)$. 
\end{proof}
Thus the local diameter Gehring-Hayman is \textit{a priori} a stronger condition than the length counterpart, in the sense of the above lemma when the points $x$ and $y$ are chosen correctly.
In the next definition we introduce the class of domains for which we show the $W^{k,\infty}$-density to hold.

\begin{defn}[$(c_0,c,R)$-radially hyperbolic]\label{approx_domain}
 Let $c_0,c\geq 1,\;R\geq 0$. We say that a domain $\Omega\subsetneq\R^n$ is $(c_0,c,R)$-radially hyperbolic with center $x_0$, if it has the $c_0$-ball separation property and there exists $x_0\in\Omega$ so that for any $x_0\neq x\in\Omega$ and any quasihyperbolic geodesic $\Gamma_x$ joining $x_0$ to $x$, the following are true: 
 \begin{enumerate}
 
  \item[(\textit{i})]
  $\Gamma_x$ has the $(c,R)$-diameter Gehring-Hayman property (from Definition \ref{local_GeHa}).
  
  \item[(\textit{ii})]
  Whenever $x_0\neq y\in\Omega$ and $\Gamma_y$ is a quasihyperbolic geodesic from $x_0$ to $y$ such that $(\Gamma_x\cap \Gamma_y)\setminus\{x_0\}\neq \emptyset$, the set $\Gamma_x\cup \Gamma_y$  satisfies either of the two: 
  \begin{enumerate}
  \item[(a)] the $(c,R)$-length Gehring-Hayman property,  
  \item[(b)] the $(c,R)$-diameter Gehring-Hayman property. \end{enumerate} 
 \end{enumerate}
\end{defn}
We will refer to the curves $\Gamma_x$ for $x_0\neq x\in\Omega$ as radial geodesics. We will say $\Omega$ is $(c_0,c,R)$-radially hyperbolic when reference to the center $x_0$ is not required or when $x_0$ is fixed and understood. 
\begin{rem}
\begin{enumerate}We note the following.
\item[(\textit{i})]It is clear that a $(c_0,c,R)$-radially hyperbolic domain is also $(c_0,c,R')$-radially hyperbolic whenever $R>R'$.
\item[(\textit{ii})]When the geodesics are unique the second requirement of above definition becomes vacuous. An example of this class would have to be simply connected as noted previously; a planar simply connected domain for example.
\item[(\textit{iii})]\label{equiv_1}If we have that the $(c,R)$-diameter Gehring-Hayman property holds for all relevant pairs of points in the second requirement, then we get the first requirement as an implication of it. This need not be the case in general, and therefore we specify the first requirement separately.
\item[(\textit{iv})]It follows from Theorem \ref{Gromovhyperbolic_implies} and Lemma \ref{diameter_GeHa} below that a $\delta$-Gromov Hyperbolic domain is a $(c_0,c,\infty)$-radially hyperbolic domain (for any choice of center), where $c=c(\delta,n)$ and $c_0=c_0(\delta,n)$. We do not know if the converse is true as the Gehring-Hayman conditions above are required to hold only radially.  Also note that when $R$ is close enough to zero, say $R<1/2$, the domains defined above may only have the ball separation property, as the local Gehring-Hayman conditions above become vacuous. We would like to know if for some $R\in(0,\infty]$, $(c_0,c,R)$-radially hyperbolic domains are Gromov hyperbolic.  
 
\end{enumerate}
\end{rem}


\subsection{Approximating polynomial}\label{approx_poly}
Let $v\in W_{loc}^{k,p}(\Omega)$.
For each measurable set $E\Subset\Omega$, let us denote by $P_E=P(v;k,E)$ the unique polynomial of order $k-1$ such that 
$$
\int_E \nabla^{\alpha}(v-P_E)=0,
$$ for all multi-indices $\alpha$ such that $|\alpha|\leq k$ (see Jones \cite[page 79]{J}).

We note the following general result.
\begin{lem}[Lemma 2.1, \cite{J}] \label{norm_equivalence}
Let $Q$ be any cube in $\R^n$ and $P$ be a polynomial of degree $k$ defined in $\R^n$. Let $E,F\subset Q$ be such that $|E|,|F|>\eta |Q|$ where $\eta>0$. Then
$$\|P\|_{L^p(E)}\leq C(\eta, k)\|P\|_{L^p(F)}.$$
\end{lem}

Let $Q,Q'\in\mathcal{W}$ be Whitney cubes where $\mathcal{W}$ is the Whitney decomposition of a domain $\Omega\subset\R^n$, such that there is a chain of $N_0$ Whitney cubes $\{Q=Q_1,Q_2,\ldots,Q_{N_0}=Q'\}\subset\mathcal{W}$ forming a continuum joining $Q$ and $Q'$ such that consecutive cubes from the collection intersect in faces. Then we have the following estimate from \cite{J}, which follows via chaining, using the Poincar\'e inequality.

\begin{lem}[Lemma 3.1, \cite{J}]\label{chaining}
 Fix $\alpha$ such that $|\alpha|\leq k$ and let $v\in W_{loc}^{k,p}(\Omega)$. Then for any pair of Whitney cubes as above, we have 
 $$
 \|\nabla^{\alpha}(P_Q-P_{Q'})\|_{L^p(Q)}\leq C(n, N_0)l(Q)^{k-|\alpha|}\|\nabla^kv\|_{L^p(\cup_{i=1}^{N_0}Q_i)}
 $$
 where $|\nabla^kv|$ is the $\ell_2$-norm of the vector $\{\nabla^{\alpha}v\}_{|\alpha|=k}$. 
\end{lem}

\section{Density in the $L^{k,p}$-norm}\label{homogeneous_density}
                                                                                                       
In this section we prove the following density result.                                                                                                    
 \begin{thm}\label{main_1}
 Let $\Omega\subset \R^n$ be a bounded $c_0$-ball separation domain. Then, given $c\geq 1$, there exists $R_0=R_0(c_0,c,n)$ such that, if $\Omega$  is $(c_0,c,R)$-radially hyperbolic for any $R\geq R_0$, we have that $ C^{\infty}(\Omega)\cap W^{k,\infty}(\Omega)$ is dense in $L^{k,p}(\Omega)$.
\end{thm}

Let $\epsilon>0$ be fixed. Let $u\in L^{k,p}(\Omega)$ be given. We will define a function $u_{\epsilon}\in W^{k,\infty}(\Omega)\cap L^{k,p}(\Omega)\cap C^{\infty}(\Omega)$ such that $\|u-u_{\epsilon}\|_{L^{k,p}(\Omega)}\lesssim\epsilon$.
We begin by constructing a suitable decomposition of $\Omega$ into subsets with finite overlap.

 

\subsection{A decomposition of $\Omega$}\label{Whitney's}
In this section let $\Omega\subsetneq \R^n$ be a domain with the $\frac{c_0}{10}$-ball separation property for $c_0\geq 10$.
We denote by $\mathcal{W}$ the Whitney decomposition of the domain $\Omega$.

We briefly describe the role of the ball separation condition in the decomposition of $\Omega$. Recall that we want to extend the function from a connected compact `core' (where a smooth aproximation of the function $u$ to be approximated has bounded derivatives of order up to $k$) formed by large Whitney cubes to the rest of the domain in a such a way that the derivatives of the extension are still bounded. Suitable neighbourhoods (to be determined by the separation property) of cubes in the boundary of the core block parts of the remaining domain which appear as `tentacles' separated from other tentacles and the core (see Figure \ref{tentacle}). The points in a tentacle should then receive their values for the extension of the given function from the cube whose neighbourhood separates it from the core and the other tentacles. However, (it may be seen below that) we need to deal with the case when there are more than one possible choices of cubes which block a given such tentacle. In particular, we need the polynomial approximations (Section \ref{approx_poly}) to our function, in cubes whose neighbourhoods intersect, to oscillate in a controlled way with respect to each other, so that the derivatives are not very large. The construction in \cite{KRZ} takes care of this by considering intrinsic neighbourhoods of cubes for blocking the tentacles. The Gehring-Hayman condition in that case ensures that if suitable intrinsic neighbourhoods of two Whitney cubes intersect, for example with a tentacle, then the euclidean length of the quasihyperbolic geodesic joining their centres is at most, say a constant multiple of the diameter of the larger cube. This is then used to obtain suitable estimates for the oscillations. We will however need to work with euclidean neighbourhoods to define a euclidean partition of unity. This is because the intrinsic distance is only (locally) Lipschitz and we need a partition of unity with $k$-derivatives. Thus we modify the construction in \cite{KRZ}, allowing us also to work under slightly weaker hypothesis. We do this below.


Let $\Omega_m^{(1)}$ denote the interior of the connected component containing $x_0$ of the set $\cup\{Q:l(Q)\geq 2^{-m}\}$, where $m$ is large enough so that $\Omega_m^{(1)}$ is well defined (contains the point $x_0$). 
Define
$$
 \mathcal{W}_m^{(1)}:=\{Q\in\mathcal{W}:Q\subset \Omega_m^{(1)}\}
$$
and
$$
\mathcal{P}_m^{(1)} :=\{Q\in\mathcal{W}_m^{(1)}:2^{-m}\leq l(Q)<2^{-(m-2)}\}.
$$
Given $Q\in\mathcal{W}$, $\lambda>0$, $\lambda Q$ is the concentric cube with edge length $\lambda$ times that of $Q$. We write $(\lambda Q)_c$ for the component of $\lambda Q\cap \Omega$ that contains $Q$.
Set $$B_Q=\left(\frac{11}{10}c_0Q\right)_c.$$

Given a Whitney cube $Q$ and a set $A$ in $\Omega$, we say that $Q$ \textit{blocks} $A$ if $x_0$ and $A$ lie in different components of $\Omega\backslash \overline{(c_0Q)}_c$. Given another Whitney cube $Q'$, we write $Q|Q'$, if $Q$ blocks the set $B_{Q'}$. Set 
$$\mathcal{P}_m^{(-1)}:=\left\{Q\in\mathcal{P}_m^{(1)}:\exists Q'\in\mathcal{P}_m^{(1)}\;\text{such that}\;  Q'|Q\right\}
$$
and subtract this collection from the original collection to define 
$$\mathcal{P}_m=\mathcal{P}_m^{(1)}\backslash \mathcal{P}_m^{(-1)}.$$
Let us denote the components of $\Omega\backslash \underset{Q\in\mathcal{P}_m}\cup \overline{(c_0Q)_c}$ by $V_i'$ for $i=0,1,2\ldots$ where $V_0'$ is the component containing $x_0$. Set
$$\mathcal{V}_i':=\{Q\in \mathcal{P}_m:\overline{V}_i'\cap \overline{(c_0Q)_c}\neq\emptyset \}$$ for $i=0,1,2,\ldots$ \newline
The member cubes of $\mathcal{V}_i'$ together separate the set $V_i'$ from $x_0$.
We want the sets $V_i'$ for $i>0$ to correspond to the tentacles mentioned above while $V_0'$ corresponds to the core. After that, we would like to find bounds on oscillations of the polynomial approximations to our functions across the cubes in $\mathcal{V}_i$ for each $i>0$. Some of the sets $V_i'$ however may be `pockets' composed of large cubes bounded by neighbourhoods of the cubes in $\mathcal{P}_m$. We need to relabel the sets $V_i'$ in order to exclude such `thick' components and consider them as part of the core instead.

\begin{lem}\label{relabel}

We have a decomposition $$\Omega\backslash \underset{Q\in\mathcal{P}_m}\bigcup \overline{(c_0Q)_c}=\bigcup_{i=0}^{l_m}U_i\;\cup\;\bigcup_{i\geq 1} V_i,$$ where the sets $V_i'$ have been relabeled as $U_i$ or $V_i$ so that the following hold.

 \begin{enumerate}
 \item[(\textit{i})]There are finitely many $V_i'$, denoted and enumerated as \newline $\{U_0=V_0',\ldots,U_{l_m}\}$ having the property that for any Whitney cube $Q$ such that $Q\cap U_{j}\neq \emptyset$ for $j=1,2,\ldots,l_m$ it holds that $l(Q)\geq 2^{-(m-2)}$. We write $\mathcal{U}_i:=\{Q\in \mathcal{P}_m:\overline{U}_i\cap \overline{(c_0Q)_c}\neq\emptyset \}$ for the collection of boundary cubes whose neighbourhoods bound $U_i$ for $i=0,\ldots,l_m$.
 
 \item[(\textit{ii})]The components $V_i'\notin\{U_0,\ldots,U_{l_m}\}$ are relabeled as $\{V_i\}_i$. For each $i$, set $\mathcal{V}_i:=\{Q\in \mathcal{P}_m:\overline{V}_i\cap \overline{(c_0Q)_c}\neq\emptyset \}$ as the boundary cubes whose neighbourhoods bound the set $V_i$. We have that $\#\mathcal{V}_i\leq M$, where $M=M(c_0,n)$.

 


 \end{enumerate}
  \end{lem}
\begin{proof}
We begin with the component $V_0'$, containing the point $x_0$. Suppose that there exists a point $z_0\in Q\subset V_0'$ such that $l(Q)< 2^{-(m-2)}$. Then there exists $Q_0\in\mathcal{P}_m^{(1)}$ for which $\Gamma_{z_0}\cap Q_0\neq\emptyset$.  By the ball separation property, either $x_0$ and $z_0$ lie in different components of $\Omega\setminus\overline{(c_0Q_0)}_c$ or $z_0\in\overline{(c_0Q_0)}_c$. Suppose first that $Q_0\in\mathcal{P}_m$. This gives a contradiction to our assumption that $z_0$ and $x_0$ lie in the same component of $\Omega\setminus \underset{Q\in\mathcal{P}_m}\bigcup \overline{(c_0Q)}_c$. If $Q_0\notin \mathcal{P}_m$ then there exists $Q_0^0\in\mathcal{P}_m$ such that $Q_0^0|Q_0$. Since by the ball separtion property either $\overline{(c_0Q_0)}_c$ separates $x_0$ and $z_0$ or $z_0\in\overline{(c_0Q_0)}_c$ and since $z_0\notin \overline{(c_0Q_0^0)}_c$, we have that $z_0$ and $x_0$ are in different components of $\Omega\setminus \underset{Q\in\mathcal{P}_m}\bigcup \overline{(c_0Q)}_c$ which is the same contradiction again. Therefore, there is no point $z_0\in V_0'$ as above. From this we conclude that for every $Q\in\mathcal{W}$ such that $Q\cap V_0'\neq \emptyset$ we have $l(Q)\geq 2^{-(m-2)}$. We set $U_0=V_0'$ and $\mathcal{U}_0=\mathcal{V}_0'$.
 
We continue by induction. Suppose the sets $V_j'$, where $0\leq j\leq i-1$, have been relabelled.
Consider now the component $V_i'$ of $\Omega\setminus \underset{Q\in\mathcal{P}_m}\bigcup \overline{(c_0Q)}_c$. We have the following cases.\newline \underline{Case 1:} 
\newline Suppose first that there exists $z_i\in V_i'$ and $Q_i\in \mathcal{P}_m^{(1)}$ such that $\Gamma_{z_i}\cap Q_i\neq\emptyset$. It follows from the ball separation property that $x_0$ and $z_i$ are not in the same component of $\Omega\setminus\overline{(c_0Q_i)}_c$.  We consider the following subcases.

\underline{Case 1.1} \;$\overline{(c_0Q_i)}_c\cap V_i'=\emptyset$.
 \newline Given $Q\in\mathcal{V}_i'$, such that $\overline{(c_0Q_i)}_c\cap \overline{(c_0Q)_c}\neq\emptyset$, we have that $x_0$ and $\overline{(c_0Q)}_c$ are in separate components of $\Omega\setminus\overline{(c_0Q_i)}_c$. By the definition of $\mathcal{P}_m$ it is impossible that $Q_i|Q$. We conclude that $\overline{(c_0Q_i)}_c\cap B_Q\neq\emptyset$.

\underline{Case 1.2:}\; $\overline{(c_0Q_i)}_c\cap  V_i'\neq \emptyset$.
\newline In this case $Q_i\notin \mathcal{P}_m$. So let $Q_i^i\in\mathcal{P}_m$ be such that $Q_i^i|Q_i$. Then we are back to the previous case. We conclude again that if $Q\in\mathcal{V}_i'$, then $\overline{(c_0Q_i^i)}_c\cap B_Q\neq\emptyset$.

Since the cubes in $\mathcal{V}_i'$ have comparable sizes, are disjoint and by the above reasoning are contained in a set of measure bounded from above by a constant times the measure of the smallest cube, we get $\#\mathcal{V}_i'\leq M$, where $M=M(c_o,n)$. We set $V_{j_i+1}=V_i'$ and $\mathcal{V}_{j_i+1}=\mathcal{V}_i'$ where $j_i$ is the smallest index such that $V_{j_i}$ has been already defined, (if no such $j_i$ exists, take $j_i=0$).
\newline\underline{Case 2:}
\newline Suppose next that it is not possible to find a point $z_i$ as above. Then we have that for any Whitney cube $Q$ such that $Q\cap V_i\neq\emptyset$, $l(Q)\geq 2^{-(m-2)}$, since if $2^{-(m-2)}> l(Q)$ then we get a contradiction with the assumption that there is no point $z_i$ as above. In this case we set $U_{j_i+1}:=V_i'$ and $\mathcal{U}_{j_i+1}:=\mathcal{V}_i'$, where $j_i$ is the largest index for which $U_{j_i}$ has already been defined. 

This concludes the relabelling.
\end{proof}


The sets $\overline{(c_0Q)}_c$ for $Q\in\mathcal{P}_m$, $\overline{V}_i$ and $\overline{U}_j$ provide a decomposition of $\Omega$.
Before we may proceed however, we are faced with the issue that uniformly enlarged neighbourhoods of the tentacles $\{V_i\}_i$ might not have bounded overlap which we require if we have them as sets being used to define a partition of unity. Therefore, in what follows we make some modifications to the decomposition of $\Omega$ in order to take care of this problem. Namely, we group together some of the sets $V_i$ and the neighbourhoods $B_Q$ of the cubes that block them. This provides a decomposition where suitably uniformly enlarged neighbourhoods of the sets in the decomposition have bounded overlap. We need a lemma first.

 Write $\mathcal{V}^{(m)}$ for the union $\underset{i\geq1}\bigcup \mathcal{V}_i$ (see part (\textit{ii}) of the definition of the relabeled sets in Lemma \ref{relabel}). Fix an enumeration $\{Q_1,\ldots,Q_{j_m}\}$ of $\mathcal{P}_m$.

 \begin{lem}\label{bound_belong}
  For each $Q_j\in\mathcal{P}_m$, there exists a collection $\{\mathcal{V}_{j_1},\ldots,\mathcal{V}_{j_{l}}\}$ so that $Q_j\in\mathcal{V}_{j_s}$ for $1\leq s \leq l$, which is maximal in the sense that $\mathcal{V}_{j_s}\not\subset \cup _{i\neq s} \mathcal{V}_{j_i}$ for $1\leq s \leq l$ and $Q_j\in\mathcal{V}_k$ implies $\mathcal{V}_k\subset \bigcup _{1\leq i\leq s} \mathcal{V}_{j_i}$. Moreover, $l$ is bounded above by a constant depending only on $c_0$ and $n$. 
 \end{lem}
\begin{proof}
 If $Q_j\notin\mathcal{V}^{(m)}$, then set $l=0$. If $Q_j\in\underset{1\leq s\leq l}\bigcap\mathcal{V}_{j_s}$, where $\{\mathcal{V}_{j_s}\}_{1\leq s \leq l}$ is a maximal distinct such collection, then for each $Q\in \underset{1\leq s\leq l}\bigcup\mathcal{V}_{j_s}$, $Q\subset (CQ_j)_c$, where $C=C(c_o,n)$, by part (\textit{ii}) of our decomposition in Lemma \ref{relabel}. Thus, there are an absolutely bounded number of cubes in $\underset{1\leq s\leq l}\cup\mathcal{V}_{j_s}$. So we get an upper bound for $l$ in terms of $c_0$ and $n$. To see the existence of such a maximal collection one observes that there are finitely many distinct collections $\mathcal{V}_i$.
\end{proof}
For each $Q_i\in\mathcal{P}_m$, set 
$$\tilde{\mathcal{Q}}_i= \underset{1\leq j\leq l}\bigcup\mathcal{V}_{i_j},$$ where $\{\mathcal{V}_{i_j}\}_j$ is a chosen collection for $Q_i$, maximal in the sense of the Lemma \ref{bound_belong}. Note that the definition of $\tilde{\mathcal{Q}}_i$ does not depend on the maximal collection $\mathcal{V}_{i_1},\ldots,\mathcal{V}_{i_{l}}$ used.
Choose a maximal subcollection $\{\mathcal{Q}_i\}_i$ from $\{\tilde{\mathcal{Q}}_i\}_i$ such that collections in $\{\mathcal{Q}_i\}_i$ are distinct in the sense that
$$\underset{Q\in\mathcal{Q}_i}\bigcup  Q\not\subset  \underset{i\neq s}\bigcup\underset{Q\in\mathcal{Q}_s}\bigcup Q$$ and satisfy
$$\bigcup_i\mathcal{Q}_i  =\bigcup_i\tilde{\mathcal{Q}_i}=\mathcal{V}^{(m)}.
$$
Denote this collection by $\mathcal{Q}^{(m)}:=\{\mathcal{Q}_i\}_i$. 

Note that given $V_i$, there exists $j_i\in\mathbb{N}$ such that $V_i$ and $x_0$ are in separate components of $\Omega\setminus\underset{Q\in\mathcal{Q}_{j_i}}\bigcup \overline{(c_0Q)}_c$ (since $\mathcal{V}_i$ is contained in one of these collections). There may be more than one $\mathcal{Q}_{j_i}$ satisfying the above condition, so we assign to $V_i$ the smallest possible index $j_i$ so that $\mathcal{Q}_{j_i}$ works.
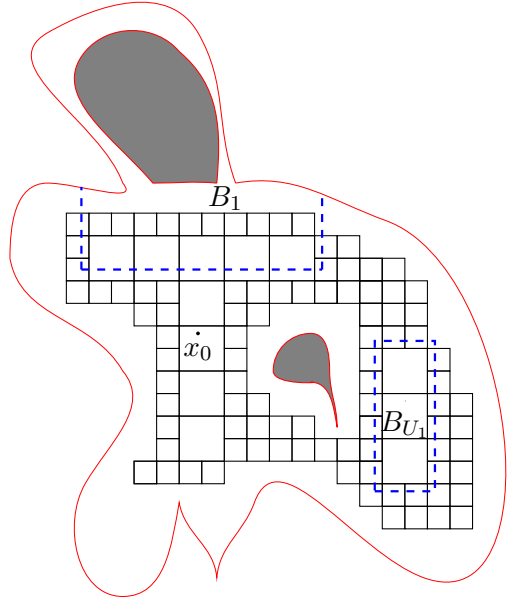
\begin{figure}
\begin{center}
\begin{tikzpicture}[scale=0.5]
\draw [red, fill=gray,thin] (0,0) to [out=90,in=180] (1,1)
to [out=0,in=90] (1.5,0) to [out=-90,in=90] (1.7,-1.5)
to [out=100,in=0] (1.0,-.3) to [out=0,in=-90] (0,0);
\draw [red,line width=0 mm](-1,5) [out=370,in=160] to (3,4) [out=340,in=340] to (4,-5.5) [out=160,in=70] to (-0.5,-3.5) [out=250,in=85] to (-1.5,-5.5) [out=95,in=280] to (-2.5,-3.5) [out=260,in=0] to (-4,-6) [out=180,in=240] to (-5,-3) [out=60,in=310] to (-4,0) [out=120,in=280] to (-7,3) [out=90, in=180] to (-6,5) [out=0,in=310] to (-4,5) [out=110,in=210] to (-6,9) [out=30,in=120] to (-2,9) [out=300,in=110] to (-1,5);
\draw [red,fill=gray,thin] (-1.5,5) [out=85,in=320] to (-2.5,8.5) [out=130, in=50] to (-5,8.5) [out=230,in=130] to (-3.2,5) [out=0,in=175] to (-1.5,5);

\foreach \y in {-2.4,-1.2,...,2.6}
 \draw [](-2.5,\y) rectangle (-1.3,\y+1.2);
 
\filldraw (-2,1) circle [radius=1pt]  node [anchor=north] {$x_0$};
\filldraw (-2,4.6)   node [anchor=west] {$B_1$};

\foreach \x in {-1.3,-0.1}
 \draw (\x,3.6) rectangle (\x+1.2,2.4);
 
\foreach \x in {-2.5,-3.7}
 \draw (\x-1.2,2.4) rectangle (\x,3.6);
 
\foreach \y in {-2.4,-1.8,...,1.8}
 \draw[black] (-2.5,\y) rectangle (-3.1,\y+0.6);
 
\foreach \y in {-2.4,-1.8,...,1.8}
 \draw[black] (-1.3,\y) rectangle (-0.7,\y+0.6); 
 
\foreach \x in {-0.7,-0.1,0.5}
 \draw [black](\x,1.8) rectangle (\x+0.6,2.4);
 
\foreach \x in {-3.1,-3.7,-4.3}
 \draw [black] (\x,1.8) rectangle (\x-0.6,2.4); 
 
\foreach \x in {-4.9,-4.3,...,0.5}
 \draw [black](\x,3.6) rectangle (\x+0.6,4.2);
 
 \foreach \y in {1.8,2.4,3,3.6}
 \draw [black] (-5.5,\y) rectangle (-4.9,\y+0.6);
 
\foreach \y in {-2.4,-1.8}
 \foreach \x in {-0.7,-0.1,0.5}
  \draw[] (\x,\y) rectangle (\x+0.6,\y+0.6);
  
\draw [](-0.7,-1.2) rectangle (-0.1,-0.6);  
 
\foreach \y in {-3,-1.8,-0.6}
 \draw [] (2.9,\y) rectangle (4.1,\y+1.2);
 
\foreach \y in {0.6,1.2,1.8}
 \foreach \x in {2.9,3.5}
  \draw (\x,\y) rectangle (\x+0.6,\y+0.6); 
  
\foreach \x in {1.1,1.7,2.3}
 \draw [] (\x,1.8) rectangle (\x+0.6,2.4);
 
 \foreach \y in {1.8,2.4,3}
 \draw [black] (1.1,\y) rectangle (1.7,\y+0.6);

\foreach \x in {1.7,2.3,2.9}
 \draw (\x,2.4) rectangle (\x+0.6,3);
 
\foreach \x in {1.7}
 \draw (\x,3) rectangle (\x+0.6,3.6); 
 
\foreach \x in {1.1,1.7,2.3}
 \draw (\x,-2.4) rectangle (\x+0.6,-1.8); 
 
\draw (2.3,-1.8) rectangle (2.9,-1.2); 

\draw (2.3,-3) rectangle (2.9,-2.4);

\draw (2.3,1.2) rectangle (2.9,1.8);
\draw (-0.7,1.2) rectangle (-0.1,1.8);
\draw (1.7,-3) rectangle (2.3,-2.4);
\draw (-3.1,1.2) rectangle (-3.7,1.8);

\foreach \x in {-1.3,-1.9,...,-3.1}
 \draw[black] (\x,-3) rectangle (\x-0.6,-2.4);
 
 \foreach \x in {-3.1}
 \draw[] (\x,-3) rectangle (\x-0.6,-2.4);
 
\foreach \y in {-3,-2.4,...,0.6}
 \draw (4.1,\y) rectangle (4.7,\y+0.6);
 
\foreach \y in {-3,-2.4,...,-0.6}
 \draw (4.7,\y) rectangle (5.3,\y+0.6); 
 
\foreach \x in {2.3,2.9,...,4.7}
 \draw (\x,-3) rectangle (\x+0.6,-3.6);
 
\foreach \x in {2.9,3.5,...,5.3}
 \draw (\x,-3.6) rectangle (\x+0.6,-4.2); 
 
 \draw[line width=0.30mm,dashed,blue] (-5.1,2.7) -- (-5.1,4.9);
 \draw[line width=0.30mm,dashed,blue] (1.3,2.7) -- (1.3,4.8);
 \draw[line width=0.30mm,dashed,blue] (-5.1,2.7) -- (1.3,2.7);
 
 \foreach \y in {-2.4,-1.8,...,2.4}
 \draw (2.3,\y) rectangle (2.9,\y+0.6);
 \draw [line width=0.30mm,dashed,blue](2.7,-3.2) rectangle (4.3,0.8);
 \filldraw (3.5,-0.8) circle [radius=0pt]  node [anchor=north] {$B_{U_1}$};
 
\end{tikzpicture}
\end{center}
\caption{Illustrating the sets $B_i$.}
\label{tentacle}
\end{figure} 

Therefore, we may classify the collection of components $\{V_i\}_i$ based on who blocks them as follows. Recall that $B_Q=(\frac{11}{10}c_0Q)_c$. For each $\mathcal{Q}_i\in\mathcal{Q}^{(m)}$, let $\mathcal{B}_{i}$ be the collection $\{V_{i_j}\}_j$ of components that have been assigned as above to $\mathcal{Q}_i$. Set
$$
B_i:=\underset{V_{i_j}\in\mathcal{B}_i}\bigcup \overline{V}_{i_j}\;\cup\; \underset{Q\in\mathcal{Q}_i}\bigcup B_Q.$$
The sets $B_i$ may be now interpreted as the tentacles mentioned at the beginning of this section (see Figure \ref{tentacle}). Fix a cube $Q_i\in\mathcal{Q}_i$, which we will below refer to as the `assigned' cube of $\mathcal{Q}_i$ (and $B_i$). 
We write $\mathcal{B}^{(m)}$ for the collection $\{\mathcal{B}_i\}_i$.

Set $\mathcal{U}^{(m)}=\mathcal{P}_m\setminus \bigcup\{\mathcal{Q}_i\in\mathcal{Q}^{(m)}\}$ as the collection of cubes whose neighbourhoods are not needed for separating the components $V_i$.
Next, we describe the decomposition of $\Omega$ which we use for our partition of unity. The sets are 
\begin{enumerate}
\item[(\textit{i})]$B_{U_i}:=B(U_i, 2^{-m}/100)$, $0\leq i\leq l_m$,
\item[(\textit{ii})]$B_i$, where $\mathcal{Q}_i\in\mathcal{Q}^{(m)}$ and
\item[(\textit{iii})]$B_Q$, where $Q\in\mathcal{U}^{(m)}$.
\end{enumerate}

 Then we have the following estimate.
\begin{lem}It holds that
\begin{equation}\label{bounded_overlap}
 1\leq \sum_{Q\in\mathcal{U}^{(m)}}\ind_{B_Q}(x)+\underset{0\leq i\leq l_m}\sum\ind_{B_{U_i}}(x)+\underset{\mathcal{Q}_i\in\mathcal{Q}^{(m)}}\sum\ind_{B_{i}}(x)\leq c'(c_0,n),
\end{equation}for $x\in\Omega$.
\end{lem}
\begin{proof}
The lower bound follows since $\Omega=\underset{0\leq i\leq l_m}\bigcup B_{U_i}\cup \underset{Q_i\in\mathcal{Q}_i}\bigcup B_i\cup\underset{Q\in\mathcal{U}^{(m)}}\bigcup B_Q$. For the upper bound, it is clear that
 $$\sum_{i=0}^{l_m}\ind_{B_{U_i}}\leq c'(n)\;\,\text{and}\;\sum_{Q\in\mathcal{U}^{(m)}}\ind_{B_Q}\leq c'(c_0,n).$$  
Let $x\in\Omega$. Suppose $x\in B_{i_i}\cap\ldots\cap B_{i_l}$, where $\{B_{i_j}\}_{j=1}^l$ is the collection of all such sets $B_i$. We want an upper bound for $l$. Then, since $\{V_j\}_j$ are pairwise disjoint, for any $Q\in\underset{i_1\leq s\leq i_l}\cup\mathcal{Q}_s$ we have by Lemma \ref{bound_belong} and the construction of the collections $\mathcal{V}_j$ that $$\dist(x,Q)\leq C2^{-m},$$ where $C=C(c_0)$. This gives an upper bound in terms of $c_0$ and $n$ for the number of cubes $\#\underset{i_1\leq s\leq i_l}\cup\mathcal{Q}_s$. So, we get an upper bound $A$ for $l$, such that $A=A(c_0,n)$.
\end{proof}

Define $$\Omega_m=\bigcup_{j=0}^{l_m} U_j.$$We note the following. 
\begin{enumerate}
 \item[($\ast$)] We have that $\Omega_M^{(1)}\subset\Omega_m$ for $M(m,c_0)$ chosen small enough; precisely we require that, $2^{-M}>10\,c_0\,2^{-m}$. This follows from our definitions. We therefore have that $|\Omega_m|\rightarrow |\Omega|$ as $m\rightarrow\infty$. 
\end{enumerate}

We will need the following definitions for the ensuing lemmas.
\begin{defn}[\textit{Trail} and \textit{cover} of Q]Given $Q\in\mathcal{W}$, define the \textit{trail of $Q$}, denoted $\mathcal{T}(Q)$ to be the set of all points $y\in\Omega$ such that $\Gamma_y(x_0,y)\cap Q\neq \emptyset$. In particular $Q\subset\mathcal{T}(Q)$.

Next, given $Q\in\mathcal{P}_m$ define the \textit{cover of $Q$}, denoted by $\mathcal{A}(Q)$ to be the collection of those Whitney cubes $Q'\in\mathcal{W}_m^{(1)}$ such that there exists $z\in B_Q$ for which either 
\begin{enumerate}
\item[(a)] $z\in Q'$, or
\item[(b)] $z\in\mathcal{T}(Q')$, $Q'\in\mathcal{P}_m^{(1)}$ and $\Gamma_z\cap Q'\neq\emptyset$ for a radial geodesic $\Gamma_z$. 
\end{enumerate}
\end{defn}
 The following lemmas are just direct consequences of our geometric assumptions on $\Omega$ and the definition of our decomposition.

\begin{lem}\label{mainlem_0}
 Given $Q\in\mathcal{P}_m$, we have that $B_Q\subset\cup\{\mathcal{T}(Q'):Q'\in\mathcal{A}(Q)\}$ and $\#\mathcal{A}(Q)\leq c(c_0,n)$. 
\end{lem}
\begin{proof}
 Let $z\in B_Q$.
We have two cases. Let us first assume that $z\in\mathcal{T}(Q')$,  where $Q'\in\mathcal{P}_m^{(1)}$ satisfies $\Gamma_z\cap Q'\neq\emptyset$. Then by the ball separation applied to the curve $\Gamma_z$, we have that if $\overline{(c_0Q)}_c\cap \overline{(c_0Q')}_c\neq\emptyset$, then $x_0$ and $(c_0Q)_c$ are in different components of $\Omega\setminus\overline{(c_0Q')}_c$. Since $Q'|Q$ is false, we must have $\overline{(c_0Q')}_c\cap \overline{B}_Q\neq\emptyset$.
  
Next note that if $z$ is not of above type, then $z\in\Omega_m^{(1)}$ and thus case (a) in the definition of $\mathcal{A}(\cdot)$ must hold. Thus $B_Q\subset\cup\{\mathcal{T}(Q'):Q'\in\mathcal{A}(Q)\}$.

Again, since $\overline{(c_0Q')}_c\cap \overline{B}_Q\neq\emptyset$ for $Q'\in\mathcal{A}(Q)$, by volume comparison we get the required bound on the cardinality of $\mathcal{A}(Q)$. Indeed, the members of $\mathcal{A}(Q)$ are comparably large with respect to $Q$ and are all contained in the set $100c_0Q$.
\end{proof}

So far we have only used the ball separation condition. The next lemma specifies the role of the Gehring-Hayman conditions. Denote by $\dist_k(A,B)$ the distance measured in the quasihyperbolic metric between sets $A,B \subset\Omega$.

\begin{lem}\label{mainlem_1}
Given $c\geq 1$, there exists $R=R(c_0,c,n)$ such that, if $\Omega$ is $(c_0,c,R')$-radially hyperbolic for $R'\geq R$, then given $Q\in\mathcal{P}_m$, and $Q_1$ and $Q_2$ in $\mathcal{A}(Q)$ such that $\mathcal{T}(Q_1)\cap \mathcal{T}(Q_2)\neq \emptyset$, we have that $\dist_k(Q_1,Q_2)\leq M$, where $M=M(c_0,c,n)$.
\end{lem}

The proof of Lemma \ref{mainlem_1} is rather technical and we postpone it to Section \ref{Hyperbolic}.
\begin{rem}\label{simplification}
We note that if a global diameter Gehring-Hayman property or a suitable local quantitative version of it holds then the proof is simpler. In fact Lemma \ref{mainlem_2} below then follows directly from hypothesis without Lemma \ref{mainlem_1} being necessary. Indeed, if $Q_1,Q_2\in\mathcal{P}_m$, then pairs of points from $Q_1$ and $Q_2$ can be joined by a curve which is uniformly away from the boundary, namely contained in $\Omega_m^{(1)}$; and pairs of such points are comparably distant from the boundary. Lemma \ref{mainlem_2} then follows from Lemma \ref{local_dia_implies_len}.
\end{rem}

\begin{lem}\label{mainlem_2}With the above definitions the following hold:
\begin{enumerate}
 \item [(\textit{i})]For any $Q$ and $Q'$ in $\mathcal{P}_m$ such that $B_Q\cap B_{Q'}\neq\emptyset$, we have that $\dist_k(Q,Q')\leq c'(c_0,c,n)$.
  \item[(\textit{ii})]For any $Q\in\mathcal{P}_m$ and $Q'\in\mathcal{A}(Q)$ such that $B_Q\cap Q'\neq\emptyset$, we have that $\dist_k(Q,Q')\leq c'(c_0,c,n)$.
\end{enumerate}
\end{lem}
\begin{proof}
 By Lemma \ref{mainlem_0}, we have that $\#(\mathcal{A}(Q)\cup\mathcal{A}(Q'))$ is finite and bounded above by $c'=c'(c_0,n)$. Also note that $B_Q\cup B_{Q'}$ is connected and $\mathcal{T}(Q'')$ is relatively closed (by Arzela-Ascoli) and connected in $\Omega$ for each $Q''\in \mathcal{A}(Q)\cup\mathcal{A}(Q')$. Moreover, $\{\mathcal{T}(Q''):Q''\in\mathcal{A}(Q)\cup \mathcal{A}(Q')\}$ covers $B_Q\cup B_{Q'}$ efficiently (that is, each $\mathcal{T}(Q'')$ necessarily intersects $B_Q\cup B_{Q'}$). Thus we may find an enumeration of $\mathcal{A}(Q)\cup\mathcal{A}(Q')$ say $Q_1,\ldots,Q_s$; $s\leq c'$ such that $\mathcal{T}(Q_i)\cap \mathcal{T}(Q_{i+1})\neq \emptyset$ for $1\leq i<s$.
 Part (\textit{i}) of the Lemma then follows from Lemma \ref{mainlem_1}.
 
 Part (\textit{ii}) follows arguing similarly as in part (\textit{i}), since in this case also $Q$ and $Q'$ can again be connected by a chain of trails of cubes in $\mathcal{A}(Q)$ as before, consecutive members of which have intersecting trails when enumerated suitably, so that Lemma \ref{mainlem_1} applies.
\end{proof}

Recall that we fixed a cube $Q_i\in\mathcal{Q}_i$, called the assigned cube of $\mathcal{Q}_i$ and $B_i$, when $\mathcal{Q}_i\in\mathcal{Q}^{(m)}$. When defining the approximating function, the values of $u$ in $Q_i$ will be used to assign values to the larger set $B_i$.
\begin{lem}\label{mainlem_3}
If $Q\in\mathcal{Q}_i$, for $\mathcal{Q}_i\in\mathcal{Q}^{(m)}$, then $\dist_k(Q,Q_i)\leq c'$, where $c'=c'(c_0,c,n)$. 
\end{lem}
\begin{proof}
Note that $\underset{Q\in\mathcal{V}_s}\cup \overline{(c_0Q)}_c$ are connected and $\#\mathcal{V}_s$ are uniformly bounded for the sets $V_s$ (recall from property (\textit{iii}) of the relabeled sets in the decomposition). Thus, $\underset{Q\in\mathcal{Q}_i}\cup B_Q$ is connected, since $\{V_s\}_s$ are disjoint. Then the claim follows by induction and proof of Lemma \ref{mainlem_2} after noting that by Lemma \ref{bound_belong}, $\#\mathcal{Q}_i$ is bounded by a constant.
\end{proof}

For each $Q,Q'\in \mathcal{W}_m^{(1)}$ such that $\dist_k(Q,Q')\leq c$, fix a chain of Whitney cubes $\{Q=Q_1,\ldots,Q_s=Q'\}$ with $s\leq c'=c'(c)$ forming a continuum joining $Q$ and $Q'$ and enumerated such that consecutive cubes intersect in faces. Denote by $\mathcal{F}(Q,Q')$ this chain. The selection may be done so that $\mathcal{F}(Q,Q')=\mathcal{F}(Q',Q)$ for any pair of cubes $Q,Q'\in\mathcal{W}_m^{(1)}$. We will need the following lemma in the estimates below.

Given $Q''\in\mathcal{W}_m^{(1)}$, denote by $\mathcal{F}(Q'')$ the set of pairs $(Q,Q')$ where $Q''\in\mathcal{F}(Q,Q')$ and one of the following holds:
\begin{enumerate}
 \item[(\textit{i})]$Q\in\mathcal{P}_m$ and $Q'\in\mathcal{A}(Q)$ such that $Q'\cap B_Q\neq \emptyset$ or
 \item[(\textit{ii})]  $Q,Q'\in\mathcal{P}_m$ \,such that $B_Q\cap B_{Q'}\neq\emptyset$ or
 \item[(\textit{iii})]  $Q=Q_i\in\mathcal{Q}_i$ and $Q'=Q_j\in\mathcal{Q}_j$, \, where $Q_i$ and $Q_j$ are the assigned cubes of the tentacles $B_i$ and $B_j$ respectively such that $B_i\cap B_j\neq\emptyset$ or
 \item[(\textit{iv})]  $Q\in\mathcal{P}_m$ and $Q'=Q_i\in\mathcal{Q}_i$,\, where $Q_i$ is the assigned cube of the tentacle $B_i$ such that $B_Q\cap B_i\neq\emptyset$.
\end{enumerate}
Then we have the following.
\begin{lem}\label{overlap_2}
 Given $Q''\in\mathcal{W}_m^{(1)}$, $\#\mathcal{F}(Q'')\leq c'(c_0,c,n)$.
\end{lem}
\begin{proof}
 By Lemmas \ref{mainlem_2} and \ref{mainlem_3} we have that for any $(Q,Q')\in\mathcal{F}(Q'')$ $$\dist_k(Q,Q'')\vee\dist_k(Q',Q'')\leq c'(c_0,c,n).$$ Then the claim follows by volume comparison.
\end{proof}

\subsection{Proof of Theorem \ref{main_1}}

We begin by defining a partition of unity. Recall the decomposition of $\Omega$ obtained in the previous section (Figure \ref{tentacle}). 
Given $Q\in\mathcal{U}^{(m)}$, let $\hat{\psi}_Q$ be a smooth function such that 
$$\hat{\psi}_Q|_{(c_0Q)_c}\equiv 1,\; \text{spt}(\hat{\psi}_Q)\subset \frac{11}{10}c_0Q$$ 
and $|\nabla^{\alpha}\hat{\psi}_Q(x)|\leq C(\alpha)2^{m|\alpha|}$ for all $x\in\Omega$ and $0\leq|\alpha|\leq k$. 

For $\mathcal{Q}_i\in\mathcal{Q}^{(m)}$, let 
and let $\hat{\varphi}_{i}$ be a smooth function such that
$$\hat{\varphi}_{i}|_{\underset{V_j\in\mathcal{B}_i}\bigcup \overline{V}_j\;\cup\; \underset{Q\in\mathcal{Q}_i}\bigcup \overline{(c_0Q)}_c}\equiv 1,\; \text{spt}(\hat{\varphi}_{i})\subset \underset{V_j\in\mathcal{B}_i}\bigcup B(V_j,2^{-m}/100)\;\cup\; \underset{Q\in\mathcal{Q}_i}\bigcup (\frac{11}{10}c_0Q)$$ 
and $|\nabla^{\alpha}\hat{\varphi}_{i}(x)|\leq C(\alpha)2^{m|\alpha|}$ for all $x\in\Omega$ and $0\leq|\alpha|\leq k$. 

For $U_i$, $0\leq i\leq l_m$, let $\hat{\xi}_i$ be a smooth function such that
$$\hat{\xi}_i|_{U_i}\equiv 1,\; \text{spt}(\hat{\xi}_i)\subset B_{U_i}$$ and $|\nabla^{\alpha}\hat{\xi}_i|\leq C(\alpha)2^{m|\alpha|}$ for all $x\in\Omega$ and $0\leq |\alpha|\leq k$. The functions above are obtained by standard mollification of suitable indicator functions related to the sets involved.

 We define the partition of unity by setting 
$$\varphi_i(x)=\frac{\hat{\varphi_i}(x)}{\underset{Q\in\mathcal{U}^{(m)}}\sum\hat{\psi}_{Q}(x)+\sum_{i=0}^{l_m}\hat{\xi}_i(x)+\underset{\mathcal{Q}_i\in\mathcal{Q}^{(m)}}\sum\hat{\varphi}_{i}(x)} $$
and similarly defining $\psi_Q$ and $\xi_i$ by dividing by the sum as above. As a consequence of (\ref{bounded_overlap}) we have
\begin{enumerate}
 \item[(\textit{i})]$\underset{Q\in\mathcal{U}^{(m)}}\sum\psi_{Q}(x)+\sum_{i=0}^{l_m}\xi_i(x)+\underset{\mathcal{Q}_i\in\mathcal{Q}^{(m)}}\sum\varphi_{i}(x)=1$, for all $x\in\Omega$.
 \item[(\textit{ii})]$0\leq \psi_Q,\varphi_i,\xi_i\leq 1$.
 \item[(\textit{iii})]$\max\{ \|\nabla^{\alpha}\psi_Q\|_{L^{\infty}(\Omega)}, \|\nabla^{\alpha}\varphi_i\|_{L^{\infty}(\Omega)}, \|\nabla^{\alpha}\xi_i\|_{L^{\infty}(\Omega)}\} \leq C(\alpha)2^{m|\alpha|},$ whenever $0\leq |\alpha| \leq k$.
 \item[(\textit{iv})] $\text{spt}(\psi_Q)\subset \frac{11}{10}c_0Q$ for $Q\in\mathcal{U}^{(m)}$,\newline $\text{spt}(\varphi_i)\subset \underset{V_j\in\mathcal{B}_i}\bigcup B(V_j,2^{-m}/100)\;\cup\; \underset{Q\in\mathcal{Q}_i}\bigcup (\frac{11}{10}c_0Q)$ for $\mathcal{Q}_i\in\mathcal{Q}^{(m)}$ and \newline $\text{spt}(\xi_i)\subset B_{U_i}$ for $i=1,\ldots,l_m$.
\end{enumerate}

\begin{proof}[Proof of Theorem \ref{main_1}]
By the density of smooth functions in $L^{k,p}(\Omega)$ (see Maz'ya's book \cite{Mazya} for example), we may assume without loss of generality that $u\in C^{\infty}(\Omega)$. Recall the definition of approximating polynomials from Section \ref{approx_poly}. For any Whitney cube $Q$ we write $P_Q$ for the approximating polynomial of $u$ in $Q$. For $\mathcal{Q}_i\in\mathcal{Q}^{(m)}$, recall $Q_i\in\mathcal{Q}_i$ as the assigned cube for $\mathcal{Q}_i$ (appearing also in Lemma \ref{mainlem_3}). We write $P_i$ for the polynomial $P_{Q_i}$. We define an approximating function by setting
\begin{equation*}
 u_m(x):=\sum_{Q\in\mathcal{U}^{(m)}}\psi_Q (x)P_Q(x)+\sum_{\mathcal{Q}_i\in\mathcal{Q}^{(m)}}\varphi_{i}(x)P_i(x)+\sum_{i=0}^{l_m}\xi_i(x)u(x),
\end{equation*}for $m\in\mathbb{N}$ and for all $x\in\Omega$. We note that $u_m\in W^{k,\infty}(\Omega)\cap C^{\infty}(\Omega)$. 
By \eqref{bounded_overlap} and property (\textit{iv}) of the partition of unity it suffices to show that $m\in\mathbb{N}$ can be chosen large enough so that 
$$
\|u_m\|_{L^{k,p}(\underset{Q\in\mathcal{U}^{(m)}}\bigcup B_Q\;\cup\;\underset{\mathcal{Q}_i\in\mathcal{Q}^{(m)}}\bigcup B_i)}\leq \underset{Q\in\mathcal{U}^{(m)}}\sum\|u_m\|_{L^{k,p}(B_Q)}+\underset{\mathcal{Q}_i\in\mathcal{Q}^{(m)}}\sum\|u_m\|_{L^{k,p}(B_i)}\lesssim\epsilon.
$$

Fix a multi-index $\alpha$ such that $|\alpha|=k$. We note that $$\nabla^{\alpha-\beta}(\sum_{Q\in\mathcal{U}^{(m)}}\psi_Q+\sum_{\mathcal{Q}_i\in\mathcal{Q}^{(m)}}\varphi_i+\sum_{i=0}^{l_m}\xi_i)=0\; \text{for}\; \beta<\alpha\;\;\text{and}\;\;\nabla^{\alpha}P_Q=0,$$
where $\beta\leq\alpha$ means $0\leq\beta_i\leq \alpha_i$ for all $0\leq i\leq n$ where $\alpha=(\alpha_i)_i,\beta=(\beta_i)_i$. Let  $Q\in\mathcal{U}^{(m)}$. Then we have 
\begin{equation}\label{analyst's trick}
\begin{split}
 \nabla^{\alpha}u_m(x)
 &=  
 \sum_{\beta\leq\alpha}(\;\sum_{Q'\in\mathcal{U}^{(m)}}\nabla^{\beta}P_{Q'}(x)\nabla^{\alpha-\beta}\psi_{Q'}(x)
 \\&+
 \sum_{\mathcal{Q}_i\in\mathcal{Q}^{(m)}}\nabla^{\beta}P_i(x)\nabla^{\alpha-\beta}\varphi_{i}(x)+\sum_{i=0}^{l_m}\nabla^{\beta}u(x)\nabla^{\alpha-\beta}\xi_i(x)\;)
 \\& =
 \sum_{\beta< \alpha}(\sum_{Q'\in\mathcal{U}^{(m)}}\nabla^{\beta}(P_{Q'}(x)-P_Q(x))\nabla^{\alpha-\beta}\psi_{Q'}(x)
 \\&
 +\sum_{\mathcal{Q}_i\in\mathcal{Q}^{(m)}}\nabla^{\beta}(P_{i}(x)-P_Q(x))\nabla^{\alpha-\beta}\varphi_{i}(x)+\sum_{i=0}^{l_m}\nabla^{\beta}(u(x)-P_Q(x))\nabla^{\alpha-\beta}\xi_i(x)\;)
 \\&
 +\nabla^{\alpha}u(x)\sum_{i=0}^{l_m}\xi_i(x).
\end{split}
\end{equation}

Using \eqref{analyst's trick} and property (\textit{iii}) of the partition of unity, for $Q\in \mathcal{U}^{(m)}$ we get
\begin{equation}\label{est_0}
 \begin{split}
  \|\nabla^{\alpha}u_m\|_{L^p(B_Q)} &\leq \sum_{\beta<\alpha}2^{m(k-|\beta|)}\sum_{i=0}^{l_m}\|\nabla^{\beta}(u-P_Q)\|_{L^p(B_Q\cap B_{U_i})}\\&+  \sum_{\beta<\alpha}2^{m(k-|\beta|)} \sum_{{\substack{Q'\in\mathcal{U}^{(m)}\\B_Q\cap B_{Q'}\neq\emptyset}}}\|\nabla^{\beta}(P_{Q'}-P_Q)\|_{L^p(B_Q\cap B_{Q'})}  \\& +\sum_{\beta<\alpha}2^{m(k-|\beta|)}\sum_{{\substack{\mathcal{Q}_i\in\mathcal{Q}^{(m)}\\B_i\cap B_{Q'}\neq\emptyset}}}\|\nabla^{\beta}(P_{i}-P_Q)\|_{L^p(B_Q\cap B_{i})}  +\|(\sum_{i=0}^{l_m}\xi_i)\nabla^{\alpha}u\|_{L^p(B_Q)}\\&=:A_1+A_2+A_3+A_4,
 \end{split}
\end{equation}
where we write $A_i$ for the summands in their order of appearance.

We estimate them separately. First of all,
\begin{equation}\label{est_1}\begin{split}A_1&\lesssim \sum_{\beta<\alpha}2^{m(k-|\beta|)}\sum_{\substack{Q'\in\mathcal{A}(Q)\\Q'\cap B_Q\neq\emptyset}}\|\,\nabla^{\beta}u\,-\,\nabla^{\beta}P_{Q}\|_{L^p(Q')}\\&\lesssim \sum_{\beta<\alpha}2^{m(k-|\beta|)}\sum_{\substack{Q'\in\mathcal{A}(Q)\\Q'\cap B_Q\neq\emptyset}}\|\,\nabla^{\beta}u\,-\,\nabla^{\beta}P_{Q'}\|_{L^p(Q')}\\&+\sum_{\beta<\alpha}2^{m(k-|\beta|)}\sum_{\substack{Q'\in\mathcal{A}(Q)\\Q'\cap B_Q\neq\emptyset}}\|\,\nabla^{\beta}P_{Q'}\,-\,\nabla^{\beta}P_{Q}\|_{L^p(Q')}\\&\lesssim\sum_{\substack{Q'\in\mathcal{A}(Q)\\Q'\cap B_Q\neq\emptyset}}\sum_{Q''\in\mathcal{F}(Q,Q')}\|\nabla^k u\|_{L^p(Q'')},\end{split}\end{equation}
where in the first inequality we observed that $\{Q'\in\mathcal{A}(Q)\mid Q'\cap B_Q\neq\emptyset\}$ covers $B_Q\cap \bigcup_{i=0}^{l_m}\text{spt}(\xi_i)$. The second inequality is triangle inequality whereas in the last inequality we use the Poincar\'e inequality and Lemma \ref{chaining} respectively.


We estimate
\begin{equation}\label{est_2}
 \begin{split}
  A_2&=\sum_{\beta<\alpha}2^{m(k-\beta|)}\sum_{\substack{Q'\in\mathcal{U}^{(m)}\\B_Q\cap B_{Q'}\neq\emptyset}}\|\nabla^{\beta}P_{Q'}\,-\,\nabla^{\beta}P_{Q}\|_{L^p(B_Q\cap B_{Q'})}\\&\lesssim\sum_{\beta<\alpha}2^{m(k-|\beta|)}\sum_{\substack{Q'\in\mathcal{U}^{(m)}\\B_Q\cap B_{Q'}\neq\emptyset}}\|\,\nabla^{\beta}P_{Q'}\,-\,\nabla^{\beta}P_{Q}\|_{L^p(Q)}\\&\lesssim\sum_{\substack{Q'\in\mathcal{U}^{(m)}\\B_Q\cap B_{Q'}\neq\emptyset}}\sum_{Q''\in\mathcal{F}(Q,Q')}\|\nabla^k u\|_{L^p(Q'')},
 \end{split}
\end{equation}
where in the first inequality we applied Lemma \ref{norm_equivalence} and for the second inequality we used Lemma \ref{chaining}. Similarly,

\begin{equation}\label{est_3}
 \begin{split}
  A_3&\lesssim\sum_{\beta<\alpha}2^{m(k-\beta|)}\sum_{\substack{\mathcal{Q}_i\in\mathcal{Q}^{(m)}\\B_Q\cap B_{i}\neq\emptyset}}\|\nabla^{\beta}P_{i}\,-\,\nabla^{\beta}P_{Q}\|_{L^p(B_Q\cap B_{i})}\\&\lesssim\sum_{\beta<\alpha}2^{m(k-|\beta|)}\sum_{\substack{\mathcal{Q}_i\in\mathcal{Q}^{(m)}\\B_Q\cap B_{i}\neq\emptyset}}\|\,\nabla^{\beta}P_{i}\,-\,\nabla^{\beta}P_{Q}\|_{L^p(Q)}\\&\lesssim\sum_{\substack{\mathcal{Q}_i\in\mathcal{Q}^{(m)}\\B_Q\cap B_{i}\neq\emptyset}}\sum_{Q'\in\mathcal{F}(Q,Q_i)}\|\nabla^k u\|_{L^p(Q')}.
 \end{split}
\end{equation}
Finally,
\begin{equation}\label{est_4}
 \begin{split}
  A_4&\leq \sum_{\substack{Q'\in\mathcal{A}(Q)\\Q'\cap B_Q\neq\emptyset}}\|\nabla^k u\|_{L^p(Q')}.
 \end{split}
\end{equation}
Similarly we estimate $\|\nabla^{\alpha}u_m\|_{L^p(B_i)}$ for $\mathcal{Q}_i\in\mathcal{Q}^{(m)}$ as follows; 
\begin{equation}\label{est_0'}
 \begin{split}
  \|\nabla^{\alpha}u_m\|_{L^p(B_i)} &\leq \sum_{\beta<\alpha}2^{m(k-|\beta|)}\sum_{j=0}^{l_m}\|\nabla^{\beta}(u-P_i)\|_{L^p(B_i\cap B_{U_j})}\\&+  \sum_{\beta<\alpha}2^{m(k-|\beta|)} \sum_{{\substack{Q\in\mathcal{U}^{(m)}\\B_i\cap B_{Q}\neq\emptyset}}}\|\nabla^{\beta}(P_{Q}-P_i)\|_{L^p(B_i\cap B_{Q})}  \\& +\sum_{\beta<\alpha}2^{m(k-|\beta|)}\sum_{{\substack{Q_j\in\mathcal{Q}^{(m)}\\B_i\cap B_{j}\neq\emptyset}}}\|\nabla^{\beta}(P_{j}-P_i)\|_{L^p(B_j\cap B_{i})}  +\|(\sum_{j=0}^{l_m}\xi_j)\nabla^{\alpha}u\|_{L^p(B_i)}\\&=:A_1'+A_2'+A_3'+A_4'. 
 \end{split}
\end{equation}
Now,
\begin{equation}\label{est_1'}
\begin{split}A_1'&\lesssim \sum_{\beta<\alpha}2^{m(k-|\beta|)}\sum_{Q\in \mathcal{Q}_i}\sum_{\substack{Q'\in\mathcal{A}(Q)\\Q'\cap B_Q\neq\emptyset}}\|\,\nabla^{\beta}u\,-\,\nabla^{\beta}P_{i}\|_{L^p(Q')}\\&\lesssim \sum_{\beta<\alpha}2^{m(k-|\beta|)}\sum_{Q\in \mathcal{Q}_i}\sum_{\substack{Q'\in\mathcal{A}(Q)\\Q'\cap B_Q\neq\emptyset}}\|\,\nabla^{\beta}u\,-\,\nabla^{\beta}P_{Q'}\|_{L^p(Q')}\\&+\sum_{\beta<\alpha}2^{m(k-|\beta|)}\sum_{Q\in \mathcal{Q}_i}\sum_{\substack{Q'\in\mathcal{A}(Q)\\Q'\cap B_Q\neq\emptyset}}\|\,\nabla^{\beta}P_{Q'}\,-\,\nabla^{\beta}P_{i}\|_{L^p(Q')}\\&\lesssim\sum_{Q\in \mathcal{Q}_i}\sum_{\substack{Q'\in\mathcal{A}(Q)\\Q'\cap B_Q\neq\emptyset}}\sum_{Q''\in\mathcal{F}(Q_i,Q')}\|\nabla^k u\|_{L^p(Q'')}.\end{split}\end{equation}
Secondly,
\begin{equation}\label{est_2'}
 \begin{split}
  A_2'&\lesssim\sum_{\beta<\alpha}2^{m(k-|\beta|)}\sum_{\substack{Q\in\mathcal{U}^{(m)}\\B_Q\cap B_{Q'}\neq\emptyset}}\|\nabla^{\beta}P_{Q}\,-\,\nabla^{\beta}P_{i}\|_{L^p(B_Q\cap B_i)}\\&\lesssim\sum_{\beta<\alpha}2^{m(k-|\beta|)}\sum_{\substack{Q\in\mathcal{U}^{(m)}\\B_Q\cap B_{i}\neq\emptyset}}\|\,\nabla^{\beta}P_{Q}\,-\,\nabla^{\beta}P_{i}\|_{L^p(Q_i)}\\&\lesssim\sum_{\substack{Q\in\mathcal{U}^{(m)}\\B_Q\cap B_{i}\neq\emptyset}}\sum_{Q'\in\mathcal{F}(Q,Q_i)}\|\nabla^k u\|_{L^p(Q')}.
 \end{split}
\end{equation}
Thirdly,
\begin{equation}\label{est_3'}
 \begin{split}
  A_3'&\lesssim\sum_{\beta<\alpha}2^{m(k-|\beta|)}\sum_{\substack{Q_j\in\mathcal{Q}^{(m)}\\B_j\cap B_{i}\neq\emptyset}}\|\nabla^{\beta}P_{j}\,-\,\nabla^{\beta}P_{i}\|_{L^p(B_j\cap B_{i})}\\&\lesssim\sum_{\beta<\alpha}2^{m(k-|\beta|)}\sum_{\substack{Q_j\in\mathcal{Q}^{(m)}\\B_j\cap B_{i}\neq\emptyset}}\|\,\nabla^{\beta}P_{j}\,-\,\nabla^{\beta}P_{i}\|_{L^p(Q_i)}\\&\lesssim\sum_{\substack{Q_j\in\mathcal{Q}^{(m)}\\B_j\cap B_{i}\neq\emptyset}}\sum_{Q\in\mathcal{F}(Q_j,Q_i)}\|\nabla^k u\|_{L^p(Q)}
 \end{split}
\end{equation}
and finally,
\begin{equation}\label{est_4'}
 \begin{split}
  A_4'&\leq \sum_{Q\in\mathcal{Q}_i}\sum_{\substack{Q'\in\mathcal{A}(Q)\\Q'\cap B_Q\neq\emptyset}}\|\nabla^k u\|_{L^p(Q')}.
 \end{split}
\end{equation}

Combining \eqref{est_1}, \eqref{est_2}, \eqref{est_3}, \eqref{est_4}, \eqref{est_1'}, \eqref{est_2'}, \eqref{est_3'} and \eqref{est_4'} we get
\begin{equation}\label{est_5}
\begin{split}
 \|u_m\|_{L^{k,p}(\underset{Q\in\mathcal{U}^{(m)}}\bigcup B_Q\;\cup\;\underset{\mathcal{Q}_i\in\mathcal{Q}^{(m)}}\bigcup B_i)} &\lesssim \sum_{Q\in\mathcal{P}_m}\sum_{\substack{Q'\in\mathcal{A}(Q)\\Q'\cap B_Q\neq\emptyset}}\sum_{Q''\in\mathcal{F}(Q,Q')}\|\nabla^k u\|_{L^p(Q'')}\,\\&+\, \sum_{Q\in\mathcal{U}^{(m)}}\sum_{\substack{Q'\in\mathcal{U}^{(m)}\\B_Q\cap B_{Q'}\neq\emptyset}}\sum_{Q''\in\mathcal{F}(Q,Q')}\|\nabla^k u\|_{L^p(Q'')}\\&+\sum_{\mathcal{Q}_i\in\mathcal{Q}^{(m)}}\sum_{Q\in \mathcal{Q}_i}\sum_{\substack{Q'\in\mathcal{A}(Q)\\Q'\cap B_Q\neq\emptyset}}\sum_{Q''\in\mathcal{F}(Q_i,Q')}\|\nabla^k u\|_{L^p(Q'')}\\&+\sum_{\mathcal{Q}_i\in\mathcal{Q}^{(m)}}\sum_{\substack{Q\in\mathcal{U}^{(m)}\\B_Q\cap B_{i}\neq\emptyset}}\sum_{Q'\in\mathcal{F}(Q,Q_i)}\|\nabla^k u\|_{L^p(Q')}\\& + \sum_{\mathcal{Q}_i\in\mathcal{Q}^{(m)}}\sum_{\substack{Q_j\in\mathcal{Q}^{(m)}\\B_j\cap B_{i}\neq\emptyset}}\sum_{Q\in\mathcal{F}(Q_j,Q_i)}\|\nabla^k u\|_{L^p(Q)} .
 \end{split}
\end{equation}

Therefore, for $|\alpha|=k$
\begin{equation}\label{est_6}
 \|u_m\|_{L^{k,p}(\underset{Q\in\mathcal{U}^{(m)}}\bigcup B_Q\;\cup\;\underset{\mathcal{Q}_i\in\mathcal{Q}^{(m)}}\bigcup B_i)}\lesssim \|\nabla^k u\|_{L^p(\Omega\backslash\Omega_{\alpha m}^{(1)})},
\end{equation}where $0<\alpha=\alpha(c_0,c,n)<1$; the inequality in \eqref{est_6} follows from interchanging the order of summation in \eqref{est_5} using Lemma \ref{overlap_2}.
Hence by choosing $m$ large enough, we get the required error bound.
\end{proof}

\section{Gromov Hyperbolic and related domains}\label{Hyperbolic}
In this section we show that the hypotheses of Theorem \ref{main_1} are satisfied by Gromov hyperbolic domains in $\R^n$. For this we use the uniformization of Gromov hyperbolic domains by a conformal deformation of the quasihyperbolic metric. This idea was developed in \cite{BHK} (see also \cite{KLM}), where it was proved among other things that the domain equipped with the deformed metric is a uniform space. It was also shown that the resulting metric space is Loewner (see Definition \ref{loewner} below) when equipped with a suitable measure, compatible with the metric deformation. We use this information to prove a diameter version of the length Gehring-Hayman Theorem (Theorem \ref{diameter_GeHa}) which, in turn implies the hypothesis of Theorem \ref{main_1}. In the following we discuss the main concepts needed for our purpose.

Given a domain $\Omega\subset\R^n$, consider the metric space $(\Omega,k)$ and conformal deformations of $\Omega$ by densities, denoted $\rho_{\epsilon}$, for $\epsilon>0$ and defined as (see page 28 in \cite{BHK})
$$\rho_{\epsilon}(x):=\exp \{-\epsilon k(x,x_0)\},$$ for all $x\in\Omega$,
where $x_0\in\Omega$ is a fixed basepoint. The metric $\di_{\epsilon}$ induced by $\rho_{\epsilon}$ is defined by $$\di_{\epsilon}(x,y):=\inf \int_{\gamma}\rho_{\epsilon}ds_k,$$ where the infimum is taken over all curves joining $x$ and $y$ in the domain $\Omega$. Here the measure $ds_k$ is induced by the quasihyperbolic metric $k$, that is we have 
$$\int_{\gamma}\rho_{\epsilon}ds_k=\int_{\gamma}\frac{\rho_{\epsilon}(\gamma(t))}{d_{\Omega}(\gamma(t))}dt$$ where $\gamma$ is parametrized by arc-length in the domain $\Omega$ in the usual sense.

\begin{thm}[Proposition 4.5, \cite{BHK}]\label{uniformize}
 The conformal deformations $(\Omega,\di_{\epsilon})$ of a $\delta$-hyperbolic domain $\Omega\subset\R^n$ are bounded $A(\delta)$-uniform spaces for $0<\epsilon<\epsilon_0(\delta)$.
\end{thm}
For our purpose we fix $\epsilon=\epsilon_0(\delta)/2$ and take $x_0$ as defined in Section \ref{homogeneous_density} to be the fixed base point for the conformal deformation metric $\rho_{\epsilon}$ and denote $\rho_{\epsilon}$ by $\rho$ below. We also denote the metric $\di_{\epsilon}$ as $\di_{\rho}$ for our chice of $\epsilon$.

Theorem \ref{uniformize} says that for any domain $\Omega$ which is Gromov hyperbolic with the quasihyperbolic metric (induced by the density $1/\di_{\Omega}(\cdot)$), one can find a metric $\di_{\epsilon}$ induced by multiplying the density $1/\di_{\Omega}(\cdot)$ with a suitable weight, with which the domain $\Omega$ is a bounded and uniform metric space. Hence $\Omega$ equipped with the deformed metric $\di_{\rho}$ is also quasihyperbolic with the metric $k_{\rho}$, induced by the density $1/\di_{\rho}(\cdot)$, where $$\di_{\rho}(x)=\inf \di_{\rho}(x,y),$$ and the infimum is taken over all points in the topological boundary, denoted $\partial_{\rho}\Omega$, of the metric space $(\Omega,\di_{\rho})$; see for example Theorem 3.6 in \cite{BHK}. The boundary $\partial_{\rho}\Omega$ with the metric $\di_{\rho}$ extended to the boundary is shown to be quasisymmetrically equivalent to the Gromov boundary (equipped with its quasisymmetric gauge) of $(\Omega,k)$ in Proposition 4.13 of \cite{BHK}. The boundary is thus stretched out by the deformation to make the interior uniform.
We state next as a lemma, a fact which follows from Proposition 4.37 and Lemma 7.8 in \cite{BHK}.
\begin{lem}\label{bilipschitz}
 The metric spaces $(\Omega,k_{\rho})$ and $(\Omega,k)$ are $C(\delta)$-bilipschitz equivalent.
\end{lem}

\begin{defn}[Conformal modulus]Let $Q>1$. Let $X$ be a rectifiably connected metric space. Let $\mu$ be a Borel measure in $X$. The $Q$-modulus of a family $\mathcal{G}$ of curves in $X$ is $$\text{mod}_Q (\mathcal{G})=\text{inf} \int_X f^Q d\mu,$$ where the infimum is taken over all Borel functions $f:X\rightarrow [0,\infty]$ such that $$\int_{\gamma}f ds\geq 1$$ for all $\gamma\in\mathcal{G}$.
\end{defn}

\begin{defn}[Loewner Spaces]\label{loewner}Let $Q>1$. Let $X$ be a rectifiably connected metric space. Let $\mu$ be a Borel measure in $X$. Then $X$ is Loewner space if the function $$\varphi(t)=\text{inf}\{\text{mod}_Q(E,F;X):\Delta(E,F)\leq t\}$$ is positive for each $t>0$, where $E$ and $F$ are any non-degenerate disjoint continua in $X$ with $$\Delta(E,F)=\frac{\dist(E,F)}{\diam\;E\wedge\diam\;F}$$ and $(E,F;U)$ is the family of all curves in $U$ joining the sets $E,F\subset U.$
\end{defn}

The crucial fact from \cite{BHK} for us is that the resulting uniform space $(\Omega,\di_{\rho},\mu_{\rho})$ obtained through the deformation is $n$-Loewner. This is Proposition 7.14 in \cite{BHK}. Here $d\mu_{\rho}=\rho^n dx$. We note that conformal modulus is preserved in the deformed space.
\begin{lem}\label{modulus}
 Let $\Omega\subset\R^n$ be a $\delta$-hyperbolic domain, $\delta>0$.  Then there exists $M=M(n,\delta)$ such that for any pair of points  $x,y\in\Omega$ we have 
that $$\diam_k(\Gamma_i)\leq\log 2$$ where $\Gamma_i$ are the connected components of $\Gamma\backslash B(x,M\delta_{\Omega}(x,y))$, where $\Gamma$ is a quasihyperbolic geodesic from $x$ to $y$ in $\Omega$.
\end{lem}
\begin{proof}
 Let $(\Omega,\di_{\rho},\mu_{\rho})$ be the $A$-uniform metric measure space obtained by conformal deformation (recall $\epsilon=\epsilon(\delta)/2$ was fixed), where $A=A(\delta)$ as in Theorem \ref{uniformize}. Let $C'>1$ and consider the annulus $$A_{C'}:=B(x,2C'\delta_{\Omega}(x,y))\backslash B(x,2\delta_{\Omega}(x,y)).$$ Let $C$ be the supremum of all numbers $C'$ such that there exists a connected subcurve $\Gamma_0\in [x,y]\backslash B(x,2C'\delta_{\Omega}(x,y))$ such that
 $$\diam_k(\Gamma_0)>\log 2.$$  If $-\infty\leq C<2$, then we again have the claim with $M=6$, so let us assume that $C\geq 3$ and fix $\Gamma_0$ be as above. 
 
 Let $\lambda$ be a curve connecting $x$ and $y$ which lies inside the ball $B(x,2\delta_{\Omega}(x,y))$. We have the following modulus estimate (see for example V\"ais\"al\"a \cite{Va}). 
 \begin{equation}
  \text{mod}(\lambda,\Gamma_0,\Omega)\leq\text{mod}(B(x,2C\delta_{\Omega}(x,y))\backslash B(x,2\delta_{\Omega}(x,y)),\R^n)\leq  \frac{1}{(\log C)^{n-1}}.
 \end{equation}
Next we find a lower bound for $\text{mod}_{\rho}(\lambda,\Gamma_0,\Omega)$, where $\text{mod}_{\rho}$ denotes the conformal modulus of paths computed in $(\Omega,\di_{\rho},\mu_{\rho})$. We have by Lemma \ref{bilipschitz} that 
\begin{equation}
 \diam_{k_{\rho}}(\Gamma_0)\geq \frac{\log 2}{C(\delta)}
\end{equation} where $C(\delta)$ is the constant from Lemma \ref{bilipschitz}.
Let $x_0$ and $y_0$ be points in $\Gamma_0$ such that $k_{\rho}(x_0,y_0)=k_{\rho}(\Gamma_0)$. Since $\Gamma$ is $A$-uniform we have \begin{equation}\diam_{\di_{\rho}}(\Gamma_0)\geq \di_{\rho}(x_0,y_0)\geq (2^{\frac{1}{4A^2}}-1) \di_{\rho}(x_0)\end{equation}
and \begin{equation}\dist_{\rho}(\Gamma_0,\lambda)\leq A\,\di_{\rho}(x_0).\end{equation}
We again have by uniformity that 
\begin{equation}
 A \,\diam_{\di_{\rho}}(\lambda)\geq l_{\rho}(\Gamma)\geq \di_{\rho}(x_0,y_0).
\end{equation}
Thus we obtain from the Loewner property a lower bound as required which depends only on $\delta$. This gives an upper bound for $C$ in terms of $\delta$. The theorem follows with $M=2C+1$.
\end{proof}

\begin{thm}\label{diameter_GeHa}Let $\Omega\subset\R^n$ be a $\delta$-Gromov hyperbolic domain. Then there exists a constant $M=M(\delta,n)$ such that for any pair of points $x,y\in\Omega$ we have 
\begin{equation}
 \diam(\Gamma)\leq M\delta_{\Omega}(x,y)
\end{equation} for any quasihyperbolic geodesic $\Gamma$ joining $x$ and $y$ in $\Omega$.
\end{thm}
\begin{proof}
 Choose $M$ as in Lemma \ref{modulus}. Choose a component $\Gamma_0$ of $\Gamma$ that lies outside the ball $B(x,M\delta_{\Omega}(x,y))$. If no such component exists then the claim in the theorem follows. Let $w_0\in\Gamma_0$ be the midpoint of $\Gamma_0$ in the metric $k$. For $\Gamma_0$ we have that 
 \begin{equation}
  \diam_k(\Gamma_0)\leq \log 2
 \end{equation}
 from which it follows that 
$$\Gamma_0\subset B(w_0,\di_{\Omega}(w_0)/2).$$ If $x_0\in\Gamma_0$ is the point where $\Gamma_0$ leaves $B(x,M\delta_{\Omega}(x,y))$, then 
\begin{equation}
 \di_{\Omega}(w_0)\leq 2\di_{\Omega}(x_0)\leq 2(M\delta_{\Omega}(x,y)+\di_{\Omega}(x))\leq (2M+4)\delta_{\Omega}(x,y).
\end{equation}
 Next, let $\lambda$ be a curve in $B(x,2\delta_{\Omega}(x,y))\cap\Omega$ joining $x$ and $y$. We have by the ball separation property that 
\begin{equation}
 \dist(w_0,\lambda)\leq c_0\di_{\Omega}(w_0) \leq c_0 (2M+4) \delta_{\Omega}(x,y).
\end{equation}
The same holds for any other component that falls under this case.\newline Thus $\Gamma\subset B(x,3Mc_0\,\delta_{\Omega}(x,y)).$

\end{proof}
\begin{proof}[Proof of Theorem \ref{main_cor}]
 In this case $\Omega$ is $(c_0,c,\infty)$-radially hyperbolic for some $c_0=c_0(\delta)$ and $c=c(\delta)$ by Theorem \ref{Gromovhyperbolic_implies} and Theorem \ref{diameter_GeHa}. The claim follows by Theorem \ref{main_1}.
\end{proof}
\begin{proof}[Proof of Corollary \ref{main_cor_2}] Using the $C^0$-boundary assumption we may find a sequence of Lipschitz domains $\{\Omega_k\}_{k\in\mathbb{N}}$ such that $\Omega\subset \Omega_{k+1}\Subset \Omega_k$, for each $k\in\mathbb{N}$ (see for example Corollary 1.2 of \cite{KRZ} or Corollary 1.3 of \cite{NRS}). The proof then proceeds as the proof of Corollary 1.3 of \cite{NRS}.

\end{proof}
It remains now to prove Lemma \ref{mainlem_1}.
\begin{proof}[Proof of Lemma \ref{mainlem_1}]
 Fix $x_1\in Q_1$ and $x_2\in Q_2$ such that  $(\Gamma_{x_1}\cap\Gamma_{x_2})\setminus\{x_0\}\neq\emptyset,$ and fix $x\in\Gamma_{x_1}\cap\Gamma_{x_2}$. We have by Lemma \ref{mainlem_0} and definitions that $$\delta_{\Omega}(x_1,x_2)\lesssim_{c_0,n} \di_{\Omega}(x_1)\simeq \di_{\Omega}(x_2),$$ and therefore $R=R(c_0,n)$ may be chosen based on the previous inequality, such that if $\Omega$ is $(c,R)$-radially hyperbolic with only diameter Gehring-Hayman in requirement (\textit{ii}), then by Lemma \ref{local_dia_implies_len}, the conclusion of Lemma \ref{mainlem_1} follows. We still need to show the existence of $R$ such that also allowing the $(c,R)$-length Gehring-Hayman condition provides the claim (cf. Remark \ref{equiv_1}). Towards this end we need to show that 
 $$\lambda_{\Omega}(x_1,x_2)\lesssim_{c_0,c,n} \di_{\Omega}(x_1)\simeq \di_{\Omega}(x_2),$$ when $x_1\in Q_1$ and $x_2\in Q_2$ such that  $(\Gamma_{x_1}\cap\Gamma_{x_2})\setminus\{x_0\}\neq\emptyset$. This will be achieved through the lemmas below.
 For convenience of notation we write $\Gamma_1$ for $\Gamma_{x_1}$ and $\Gamma_2$ for $\Gamma_{x_2}$.
 
 \begin{lem}\label{sublem_1} If $z\in \Gamma_i(x_0,x_i)$ for $i=1$ or $i=2$, then 
 \begin{equation}\label{eq_1.1}
 5(c_0+1)\di_{\Omega}(z)\geq  \,\di_{\Omega}(x_i).
 \end{equation} \end{lem}
 \begin{proof}This is a consequence of the separation property. We consider competing curves, denoted $\lambda_i$ for $i=1,2$ which join $x_i$ to $x_0$ in $\Omega_m^{(1)}$. These curves exist because $\Omega_m^{(1)}$ was defined to be connected. Suppose $z\in\Gamma_1(x_0,x_1)$. Then, by the ball separation property of $\Gamma_1$, there exists a point $x'$ in $\lambda_1$ such that $$\lambda_{\Omega}(z,x')\leq c_0\di_{\Omega}(z).$$  On the other hand, since $x'\in\Omega_m^{(1)}$, we also have $d_{\Omega}(x')\geq \frac{1}{5}\di_{\Omega}(x_2)$ and the claim follows.\end{proof}

 Next, fix $y_1\in\Gamma_1$ and $y_2\in\Gamma_2$ such that 
 \begin{equation}\label{eq_1.2}
 \lambda_{\Omega}(x_1,y_2)=\underset{y\in\Gamma_2}\inf \,\lambda_{\Omega}(x_1,y)\leq c_0\di_{\Omega}(x_1)
 \end{equation}
 and 
 \begin{equation}\label{eq_1.3}
  \lambda_{\Omega}(x_2,y_1)=\underset{y\in\Gamma_1}\inf \,\lambda_{\Omega}(y,x_2)\leq c_0\di_{\Omega}(x_2),    
 \end{equation} where the upper bounds are due to the ball separation property, since $\Gamma_1$ and $\Gamma_2$ are competing quasihyperbolic geodesics (see Figure \ref{fig2}).

 We note that by the triangle inequality
\begin{equation}\label{eq_1.4}
 \di_{\Omega}(y_1)\leq (c_0+1)\di_{\Omega}(x_2)\simeq \di_{\Omega}(x_1)
\end{equation} and 
\begin{equation}\label{eq_1.5}
\di_{\Omega}(y_2)\leq (c_0+1)\di_{\Omega}(x_1)\simeq \di_{\Omega}(x_2).
\end{equation}

\begin{figure}
\centering
 \begin{minipage}{.45\textwidth}
 \centering
 \begin{tikzpicture}[scale=0.75]
  \draw (0,-1) [out=120,in=280] to (-1,1) [out=90,in=240] to (-1,2) [out=80,in=220]  to (0,4) [out=330,in=110] to (1,2) [out=270,in=100] to (1,1) [out=240,in=40] to (0,-1);
  \filldraw (0,-1) circle [radius=1pt] node [anchor=north] {$x$};
  \filldraw (-1,1) circle [radius=1pt] node [anchor=east] {$y_1$};
  \filldraw (-1,2) circle [radius=1pt] node [anchor=east] {$x_1$};
  \filldraw (0,4) circle [radius=1pt] node [anchor=south] {$x_0$};
  \filldraw (1,2) circle [radius=1pt] node [anchor=west] {$x_2$};
  \filldraw (1,1) circle [radius=1pt] node [anchor=west] {$y_2$};
  \draw [dashed] (-1,2) [out=280,in=130] to (0,1.2) [out=320,in=170] to (1,1);
   \draw [dashed] (-1,1) [out=340,in=130] to (-0.5,1.1) to (1,2); [out=320,in=170] to (1,1);
 \end{tikzpicture}
 \caption{Illustrating the curves $\Gamma_1$, $\Gamma_2$, $\Gamma_{x_1y_2}$ and $\Gamma_{x_2y_1}$.}
 \label{fig2}
 \end{minipage}%
 \begin{minipage}{0.05\textwidth}
\mbox{}\\
\end{minipage}
 \begin{minipage}{.45\textwidth}
 \centering
 \begin{tikzpicture}[scale=0.75]
  \draw (0,-1) [out=120,in=280] to (-1,1) [out=90,in=240] to (-1,2) [out=80,in=220]  to (0,4) [out=330,in=110] to (1,2) [out=270,in=100] to (1,1) [out=240,in=40] to (0,-1);
  \filldraw (0,-1) circle [radius=1pt] node [anchor=north] {$x$};
  \filldraw (-1,1) circle [radius=1pt] node [anchor=east] {$y_1$};
  \filldraw (-1,2) circle [radius=1pt] node [anchor=east] {$x_1$};
  \filldraw (0,4) circle [radius=1pt] node [anchor=south] {$x_0$};
  \filldraw (1,2) circle [radius=1pt] node [anchor=west] {$x_2$};
  \filldraw (1,1) circle [radius=1pt] node [anchor=west] {$y_2$};
  \draw [dashed] (-1,2) [] to (1,1);
  \draw [dashed] (-1,2) [out=280,in=130] to (0,1.2) [out=320,in=170] to (1,1);
  \filldraw (0,1.2) circle [radius=1pt] node [anchor=north] {$z$};
  \filldraw (0,1.5) circle [radius=1pt] node [anchor=south] {$z_{\epsilon}$};
  \filldraw (0.98,1.5) circle [radius=1pt] node [anchor=west] {$z'$};
 \end{tikzpicture}
 \caption{Illustrating the competing curves $\gamma_1^{\epsilon}$ and $\Gamma_{x_1y_2}$.}
 \label{placement}
 \end{minipage}
\end{figure}

\begin{lem}\label{sublem_2} If $y_i\in \Gamma_i(x_0,x_i)$ for $i=1$ or $i=2$, then the claim of Lemma \ref{mainlem_1} is true.\end{lem} 
\begin{proof}To this end, suppose $y_2\in\Gamma_2(x_0,x_2)$. Since we have that $\delta_{\Omega}(x_1,x_2)\lesssim \di_{\Omega}(x_1)$, it follows by the triangle inequality and \eqref{eq_1.2} that \begin{equation}\label{eq_2.1}\delta_{\Omega}(x_2,y_2)\lesssim \di_{\Omega}(x_1)\lesssim\di_{\Omega}(x_2).\end{equation}
 From the previous claim we also have 
 \begin{equation}\label{eq_2.2}
  \di_{\Omega}(x_2)\leq 5(c_0+1)\di_{\Omega}(y_2).
 \end{equation} Thus, by \eqref{eq_1.5}, \eqref{eq_2.2} and since $x_2,y_2\in\Gamma_2$, there exists $R_1=R_1(c_0)$ such that if $\Omega$ is $(c,R)$-radially hyperbolic for $R\geq R_1$, then by Lemma \ref{local_dia_implies_len}, we get $$k_{\Omega}(y_2,x_2)\leq c'(c_0,c,n),$$ which implies $\lambda_{\Omega}(y_2,x_2)\lesssim \di_{\Omega}(x_2)$, by \eqref{quasihyperbolic_growth_2}. This, together with \eqref{eq_1.2} gives
 $$\lambda_{\Omega}(x_1,x_2)\lesssim_{c_0,c,n} \di_{\Omega}(x_1)\simeq \di_{\Omega}(x_2).$$ Thus, since $\Gamma_1\cap\Gamma_2\neq\emptyset$, there exists $R_2=R_2(c_0,c,n)$ such that if $\Omega$ is $(c,R)$-radially hyperbolic for $R\geq \underset{1\leq i\leq 2}\max\,R_i$, then Lemma \ref{mainlem_1} follows by arguments in Lemma \ref{local_dia_implies_len}. Indeed, we only need to observe that there is a curve competing with $\Gamma_{x_1x_2}$ and joining $x_1$ and $x_2$ in $\Omega_m^{(1)}$.

 \end{proof}                                                                                                                                                                                                                                                                                                                                          Therefore we may assume $y_i\in \Gamma_i(x_i,x)$ for $i=1,2$.                                                                                                                                                                                                                                                                                                                                          Fix $\Gamma_{x_1y_2}$, a quasihyperbolic geodesic joining $x_1$ to $y_2$ in $\Omega$ and
and $\Gamma_{x_2y_1}$, a quasihyperbolic geodesic joining $x_2$ to $y_1$ in $\Omega$. 
                                                                                                                                                                                                                                                                                                                                                                                                                                                                                                                                                                                                                                                                               
                                                                                                                                                                                                                                                                                                                                                                                                                                                                                                                                                                                                                                                                                \begin{lem}\label{sublem_3}If $z\in\Gamma_{x_1y_2}$ then \begin{equation}\label{eq_3.1}100(c_0+1)^3\di_{\Omega}(z)\geq \di_{\Omega}(y_2).\end{equation}\end{lem}
                                                                                                                                                                                                                                                                                                                                                                                                                                                                                                                                                                                                                                                                                \begin{proof}Fix $z\in\Gamma_{x_1y_2}$.
Let $\mu_1$ denote the competing curve $\tilde{\Gamma}_1(x_1,x_0)*\Gamma_2(x_0,y_2)$ joining $x_1$ to $y_2$ (where $\tilde{\Gamma}$ denotes the reversed curve of $\Gamma$). Choose a point $z'\in \mu_1$ such that 
\begin{equation}\label{eq_3.2}
\lambda_{\Omega}(z,z')\leq c_0 \di_{\Omega}(z),
\end{equation} which we get from ball separation.
If $z'\notin \Gamma_2(x_2,y_2)$, then by \eqref{eq_1.1} and \eqref{eq_1.5}, we have that \begin{equation}\label{eq_3.3}
100(c_0+1)^2\di_{\Omega}(z')\geq \di_{\Omega}(x_1)(c_0+1)\geq \di_{\Omega}(y_2)
\end{equation} and by the triangle inequality with equation \eqref{eq_3.2} that 
\begin{equation}\label{eq_3.4}
\di_{\Omega}(z)(c_0+1)\geq\di_{\Omega}(z').\end{equation}
The claim follows by combining \eqref{eq_3.3} and \eqref{eq_3.4}.

So we assume next that $z'\in \Gamma_2(x_2,y_2)$ (see Figure \ref{placement}). We use now the fact that $y_2$ is taken to be at minimal intrinsic distance from $x_1$; see equation (\ref{eq_1.5}). Let $\gamma_1^{\epsilon}$ be a curve joining $x_1$ and $y_2$ such that \begin{equation}\label{eq_3.5}\lambda_{\Omega}(x_1,y_2)\geq 
l(\gamma_1^{\epsilon})-\epsilon\end{equation}  Choose a point $z_{\epsilon}\in\gamma_1^{\epsilon}$ such that 
\begin{equation}\label{eq_3.6}
 \lambda_{\Omega}(z,z_{\epsilon})\leq c_0\di_{\Omega}(z),
\end{equation} obtained by applying ball separation to the geodesic $\Gamma_{x_1y_2}$ with respect to the competing curve $\gamma_1^{\epsilon}$.
From equations \eqref{eq_3.2} and \eqref{eq_3.6} we have 
\begin{equation}\label{eq_3.7}
 \lambda_{\Omega}(z_{\epsilon},z')\leq 2c_0\di_{\Omega}(z)
\end{equation}
From equations \eqref{eq_3.5} and the choice of $y_2$ we have 
\begin{equation}\label{eq_3.8}
 \begin{split}
  \lambda_{\Omega}(x_1,z_{\epsilon})+\lambda_{\Omega}(z_{\epsilon},y_2)&\leq l(\gamma_1^{\epsilon})\\& \leq \lambda_{\Omega}(x_1,y_2)+\epsilon \\& \leq \lambda_{\Omega}(x_1,z')+\epsilon \\& \leq \lambda_{\Omega}(x_1,z_{\epsilon})+\lambda_{\Omega}(z_{\epsilon},z')+\epsilon
 \end{split}
\end{equation}
and therefore by equations \eqref{eq_3.6},\eqref{eq_3.7} and \eqref{eq_3.8} and passing to the limit we get
\begin{equation}\label{eq_3.9}
\begin{split}
 \di_{\Omega}(y_2)&\leq \di(z,y_2)+\di_{\Omega}(z)\\& \leq \underset{\epsilon\rightarrow 0}\lim \,(\,\di(z,z_{\epsilon})+\di(z_{\epsilon},y_2)+\di_{\Omega}(z)\,)\\& \leq (3c_0+1)\di_{\Omega}(z)
\end{split} 
\end{equation}
which was the claim.
\end{proof}
Similarly for points $z\in\Gamma_{x_2y_1}$, we get
\begin{equation}\label{eq_3.10}
 \di_{\Omega}(y_1)\leq 100(c_0+1)^3\di_{\Omega}(z).
\end{equation}

\begin{lem}\label{sublem_4} For $i=1,2$, we have
\begin{equation}\label{eq_4.1}
\di_{\Omega}(y_i)\leq 100(c_0+1)^4\di_{\Omega}(z) 
\end{equation}for any $z\in \Gamma_i(x_i,y_i)$.\end{lem}
\begin{proof}This follows by comparison with the competing curves $\tilde{\Gamma}_2(x_0,x_2)\ast\Gamma_1(x_0,x_1)\ast \Gamma_{x_1y_2}$ and $\tilde{\Gamma}_1(x_0,x_1)\ast\Gamma_1(x_0,x_2)\ast \Gamma_{x_2y_1}$ respectively for the geodesics $\Gamma_2(x_2,y_2)$ and $\Gamma_1(x_1,y_1)$ and using the ball separation property.\end{proof}
Next, we assume without loss of generality that 
\begin{equation}\label{eq_5.1}
k_{\Omega}(\Gamma_1(y_1,x))\leq k_{\Omega}(\Gamma_2(y_2,x)),
\end{equation}
and complete the proof of Lemma \ref{mainlem_1} by considering the following two possible cases.
\newline
Case 1: Assume first that 
$
\di_{\Omega}(y_i)\geq \frac{\di_{\Omega}(x_i)}{10(c_0+10)^2}
$
for either $i=1$ or $i=2$; say 
\begin{equation}\label{eq_5.2}
\di_{\Omega}(y_1)\geq \frac{\di_{\Omega}(x_1)}{10(c_0+10)^2}
\end{equation}
In this case we have $$\delta_{\Omega}(x_1,y_1)\leq \delta_{\Omega}(x_1,x_2)+\delta_{\Omega}(y_1,x_2)\lesssim \di_{\Omega}(y_1)\simeq \di_{\Omega}(x_1),$$ where the second equivalence comes from our assumption and \eqref{eq_1.3}. Thus there exists $R_3=R_3(c_0,c)$ such that if $\Omega$ is $(c,R)$-radially hyperbolic for $R\geq\underset{1\leq i\leq 3}\max \,R_i$, then by Lemma \ref{local_dia_implies_len} (which we may apply by Lemma \ref{sublem_4}), we have  $k_{\Omega}(x_i,y_i)\leq c'(c_0,c,n)$ and
\begin{equation}\label{eq_5.3}
l(\Gamma_1(x_1,y_1))\lesssim \di_{\Omega}(x_1).
\end{equation}

Now by \eqref{eq_5.3} we get
$$\lambda_{\Omega}(x_1,x_2)\leq l(\Gamma_1(x_1,y_1))+\lambda_{\Omega}(y_1,x_2) \lesssim \di_{\Omega}(x_1).$$ 
Thus there exists $R_4=R_4(c_0,c,n)$ such that if $\Omega$ is $(c,R)$-radially hyperbolic for $R\geq\underset{1\leq i\leq 4}\max \,R_i$, we have $l(\Gamma_{x_1x_2})\lesssim \di_{\Omega}(x_1)$, where $\Gamma_{x_1x_2}$ is any quasihyperbolic geodesic joining $x_1$ and $x_2$. 
It then follows by arguments similar to those in Lemma \ref{local_dia_implies_len}, that $k_{\Omega}(x_1,x_2)\leq c'(c_0,c,n)$ which is the claim.

Case 2: Next we consider the case 
\begin{equation}\label{eq_5.4}
\di_{\Omega}(y_i)< \frac{\di_{\Omega}(x_i)}{10(c_0+10)^2}
\end{equation}
for $i=1,2$. 
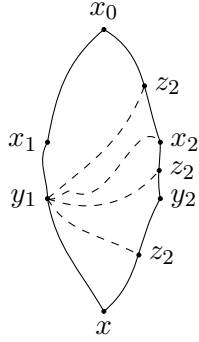
\begin{figure}
 \begin{tikzpicture}[scale=0.75]
  \draw (0,-1) [out=120,in=280] to (-1,1) [out=90,in=240] to (-1,2) [out=80,in=220]  to (0,4) [out=330,in=110] to (1,2) [out=270,in=100] to (1,1) [out=240,in=40] to (0,-1);
  \filldraw (0,-1) circle [radius=1pt] node [anchor=north] {$x$};
  \filldraw (-1,1) circle [radius=1pt] node [anchor=east] {$y_1$};
  \filldraw (-1,2) circle [radius=1pt] node [anchor=east] {$x_1$};
  \filldraw (0,4) circle [radius=1pt] node [anchor=south] {$x_0$};
  \filldraw (1,2) circle [radius=1pt] node [anchor=west] {$x_2$};
  \filldraw (1,1) circle [radius=1pt] node [anchor=west] {$y_2$};
  \draw [dashed] (-1,1) [out=340,in=130] to (-0.5,1.1) to (1,2); [out=320,in=170] to (1,1);
  \filldraw (0.98,1.5) circle [radius=1pt] node [anchor=west] {$z_2$};
  \filldraw (0.72,3) circle [radius=1pt] node [anchor=west] {$z_2$};
  \filldraw (0.62,0) circle [radius=1pt] node [anchor=west] {$z_2$};
  \draw [dashed] (-1,1) [out=340,in=230] to (0.98,1.5);
  \draw [dashed] (-1,1) [out=40,in=250] to (0.72,3);
  \draw [dashed] (-1,1) [out=300,in=150] to (0.62,0);
 \end{tikzpicture}
 \caption{Illustrating the three possibilities for the position of $z_2$.}
 \label{fig4}
\end{figure}
By the ball separation property of $\Gamma_1$, there exists a point $z_2\in\Gamma_2$ (see Figure \ref{fig4}) such that 
\begin{equation}\label{eq_5.5}
\lambda_{\Omega}(y_1,z_2)\leq c_0\di_{\Omega}(y_1).
\end{equation}
 Due to our assumption that $\di_{\Omega}(y_1)$ is small enough compared to $\di_{\Omega}(x_1)$, Lemma \ref{sublem_1} tells that $z_2\notin \Gamma_2(x_0,x_2).$ Indeed, note that
 $$\di_{\Omega}(x_2)\leq 5(c_0+1)\di_{\Omega}(z_2)\leq 5(c_0+1)^2\di_{\Omega}(y_2)
 $$ contradicts \eqref{eq_5.4}.

 We next assume that that $z_2\in\Gamma_2(y_2,x)$. Then by \eqref{eq_1.3},\eqref{eq_5.4} and \eqref{eq_5.5} we have that 
\begin{equation}\label{eq_5.6}
\lambda_{\Omega}(x_2,y_2)\leq \lambda_{\Omega}(x_2,z_2)\leq \lambda_{\Omega}(x_2,y_1)+\lambda_{\Omega}(y_1,z_2)\lesssim \di_{\Omega}(x_2).
\end{equation}
 We have by \eqref{eq_1.3} and \eqref{eq_5.6}
 \begin{equation}\label{eq_5.7}
\lambda_{\Omega}(x_1,x_2)\leq \lambda_{\Omega}(x_1,y_2)+\lambda_{\Omega}(y_2,x_2)\lesssim \di_{\Omega}(x_2).
\end{equation}
Thus there exists $R_5=R_5(c_0,c,n)$ such that if $\Omega$ is $(c,R)$-radially hyperbolic for $R\geq\underset{1\leq i\leq 5}\max \,R_i$ then $k_{\Omega}(x_1,x_2)\leq c'(c_0,c,n)$ and the claim of the lemma is true in this case. 

Next, note that if 
$$\di_{\Omega}(y_1)\leq \frac{\di_{\Omega}(y_2)}{200(c_0+1)^5}$$
and $z_2\in\Gamma_2(x_0,y_2)$, then we get a contradiction between the consequences of ball separation for $\Gamma_1$; \eqref{eq_5.5}, that 
\begin{equation}\label{eq_5.8}
\di_{\Omega}(y_1)(c_0+1)\geq \di_{\Omega}(y_1)+\lambda_{\Omega}(y_1,z_2)\geq \di_{\Omega}(z_2),\end{equation}
and that of Lemma \ref{sublem_4}. This forces $z_2\in\Gamma_2(y_2,x)$, which has been considered previously. 

We interchange the roles of the pairs $(y_1,\Gamma_1)$ and $(y_2,\Gamma_2)$ in above argument to observe that the only remaining case is when $\di_{\Omega}(y_1)\simeq_{c_0} \di_{\Omega}(y_2)$. 

Thus we only need to check the claim in the case when $z_2\in\Gamma_2(x_2,y_2)$ and $\di_{\Omega}(y_1)\simeq_{c_0} \di_{\Omega}(y_2)$. Let $\Gamma_{y_1z_2}$ be a fixed geodesic joining $y_1$ and $z_2$. Then, by \eqref{eq_5.8},
$$100(c_0+1)^4\di_{\Omega}(z_2)\geq \di_{\Omega}(y_2)\simeq \di_{\Omega}(y_1),$$
(coming from Lemma \ref{sublem_4}) and since $y_1$ and $z_2$ lie on intersecting radial geodesics, there exists $R_6=R_6(c_0)$ such that if$\Omega$ is $(c,R)$-radially hyperbolic for $R\geq\underset{1\leq i\leq 6}\max \,R_i$ we get \begin{equation}\label{eq_5.9}l(\Gamma_{y_1z_2})\lesssim \di_{\Omega}(y_1),\end{equation}                                                                                                                                                                                                                                                                                                                                                                                                                                                                                                                                                                                                                  and arguments similar to the ones in Lemma \ref{sublem_3} provide \begin{equation}\label{eq_5.10}\di_{\Omega}(y_1)\lesssim \di_{\Omega}(w)\end{equation}for all $w\in\Gamma_{y_1z_2}$. Then by a covering argument considering equations \eqref{eq_5.9} and \eqref{eq_5.10} we get $k_{\Omega}(y_1,z_2)\leq c'(c_0,c,n)$. Comparing now the quasihyperbolic lengths of the curves $\tilde{\Gamma}_{y_1z_2}\ast \Gamma_1(y_1,x)$ and $\Gamma_2(z_2,x)$ and recalling assumption \eqref{eq_5.1}, we get $k_{\Omega}(z_2,y_2)\leq c'(c_0,c,n)$ and thus
\begin{equation}\label{eq_5.11}
 \lambda_{\Omega}(z_2,y_2)\lesssim \di_{\Omega}(y_2)
\end{equation}
by \eqref{quasihyperbolic_growth_2}.
Therefore by \eqref{eq_1.2}, \eqref{eq_1.3}, \eqref{eq_5.5} and \eqref{eq_5.11} we get
\begin{equation}
 \begin{split}
  \lambda_{\Omega}(x_1,x_2) &\leq \lambda_{\Omega}(x_1,y_2)+\lambda_{\Omega}(y_2,x_2)\\& \leq \lambda_{\Omega}(x_1,y_2) + \lambda_{\Omega}(y_2,z_2)+\lambda_{\Omega}(z_2,y_1)+\lambda_{\Omega}(y_1,x_2)\\& \lesssim \di_{\Omega}(x_1)
 \end{split}
\end{equation}
Hence there exists $R_7=R_7(c_0,c,n)$ such that if $\Omega$ is $(c,R)$-radially hyperbolic for $R\geq\underset{1\leq i\leq 7}\max \,R_i$, then by  arguments in Lemma \ref{local_dia_implies_len}, we get the desired claim, namely $k_{\Omega}(x_1,x_2)\leq c'(c_0,c,n)$. This completes the proof of Lemma \ref{mainlem_1}.
\end{proof}
\section*{Acknowledgements}
The author wishes to thank his advisor Professor Pekka Koskela for suggesting this problem  and  for many discussions leading to the paper. He also wishes to thank Tapio Rajala for many valuable comments on the paper.

\end{document}